\documentclass[a4paper,leqno,11pt]{amsart}

\usepackage{amsfonts,amssymb,verbatim,amsmath,amsthm,latexsym,textcomp,amscd}
\usepackage{latexsym,amsfonts,amssymb,epsfig,verbatim}
\usepackage{amsmath,amsthm,amssymb,latexsym,graphics,textcomp}
\usepackage{mathtools}
\usepackage{paralist}
\usepackage{graphicx}
\usepackage{color}
\usepackage{url}
\usepackage{enumerate}
\usepackage[mathscr]{euscript}
\usepackage{tikz-cd}
\usetikzlibrary{shapes.geometric}
\usepackage{dsfont}
\usetikzlibrary{matrix}
\usepackage{hyperref}

\input xy

\xyoption{all}

\DeclarePairedDelimiter\floor{\lfloor}{\rfloor}

\theoremstyle{definition}
\newtheorem{theorem}{Theorem}[section]
\newtheorem{prop}[theorem]{Proposition}

\newtheorem{cor}[theorem]{Corollary}

\newtheorem{thm}[theorem]{Theorem}
\newtheorem{notation}[theorem]{Notation}
\newtheorem{subsec}[theorem]{}

\theoremstyle{plain}
\newtheorem*{thma}{Theorem A}
\newtheorem*{thmb}{Theorem B}

\theoremstyle{remark}
{\swapnumbers
   \newtheorem{ack}[theorem]{Acknowledgements} }
   
\newenvironment{myeq}[1][]
{\stepcounter{theorem}\begin{equation}\tag{\thetheorem}{#1}}
{\end{equation}}

\newenvironment{mysubsection}[2][]
{\begin{subsec}\begin{upshape}\begin{bfseries}{#2.}
			\end{bfseries}{#1}}
		{\end{upshape}\end{subsec}}

\newcommand{\C}{{\mathbb C}}

\newcommand{\Z}{{\mathbb{Z}}}
\newcommand{\R}{{\mathbb R}}

\newcommand\CC{{\mathcal C}}

\newcommand\FF{{\mathcal F}}

\newcommand\LL{{\mathcal L}}
\newcommand\MM{{\mathcal M}}

\newcommand\PP{{\mathcal P}}

\newcommand\PMF{{\PP\kern-2pt\MM\FF}}

\newcommand\PML{{\PP\kern-2pt\MM\LL}}

\newcommand{\fsubd}{\mathrel{{\scriptstyle\searrow}\kern-1ex^d\kern0.5ex}}
\newcommand{\bsubd}{\mathrel{{\scriptstyle\swarrow}\kern-1.6ex^d\kern0.8ex}}
\newcommand{\fsubeq}{\mathrel{\raise-.7ex\hbox{$\overset{\searrow}{=}$}}}
\newcommand{\bsubeq}{\mathrel{\raise-.7ex\hbox{$\overset{\swarrow}{=}$}}}

\newcommand{\tsh}[1]{\left\{\kern-.9ex\left\{#1\right\}\kern-.9ex\right\}}

\renewcommand{\Im}{\operatorname{Im}}
\newcommand{\Index}{\mbox{Index}}

\newcommand{\Ker}{\mbox{Ker}}

\newcommand{\di}{\mbox{dim}}
\newcommand{\con}{\mbox{Convex}}

\makeatletter
\@namedef{subjclassname@2020}{%
  \textup{2020} Mathematics Subject Classification}
\makeatother

\title{The index of equidimensional flag manifolds}

\author{Samik Basu}
\address{Stat-Math Unit, Indian Statistical Institute, Kolkata 700108, India.}
\email{samik.basu2@gmail.com, samikbasu@isical.ac.in}
\author{Bikramjit Kundu}
\address{Department of Mathematics, Ramakrishna Mission Vivekananda Educational and Research Institute, Belur Math, Howrah 711202}
\email{bikramju@gmail.com, bikramjit.kundu@rkmvu.ac.in}


\begin{document}
\begin{abstract}
In this paper, we consider the flag manifold of $p$ orthogonal subspaces of equal dimension which carries an action of the cyclic group of order $p$. We provide a complete calculation of the associated Fadell-Husseini index. This may be thought of as an odd primary version of the computations of Barali\'c et al \cite{BBKV18} for the Grassmann manifold $G_n(\R^{2n})$. These results have geometric consequences  for $p$-fold orthogonal shadows of a convex body. 
\end{abstract}
\maketitle

\section{Introduction}
Many combinatorial problems that rely on topological methods for their solution involve the non-existence of equivariant maps between $G$-spaces \cite{Mat03} for a finite group $G$. The spaces that arise here typically have a free $G$-action, and one computes the ``index" of such spaces to answer questions about the existence of equivariant maps out of them. There are many variants of the ``index", which are invariants of $G$-spaces. The Fadell-Husseini index \cite{Fahu88} is one of the most widely used, defined for a $G$-space $X$ as the ideal of $H^\ast(BG)$ given by
\[ \Index_G(X) = \Ker(H^\ast(BG) \to H^\ast(X_{hG})),\]
where $BG$ stands for the classifying space, $EG\to BG$ the universal $G$-bundle, and $X_{hG} = EG\times_G X$ stands for the Borel construction. The use of the index to prove the non-existence of an equivariant map $X\to Y$ lies in the condition 
\[ \Index_G(Y) \subset \Index_G(X).\]
This has been very useful in the solution of the topological Tverberg problem for prime powers \cite{Oz87}, \cite{Vol00}. For certain Stiefel manifolds, index computations in the context of Kakutani's theorem are made in \cite{BaKu21}. This is usually computed using the spectral sequence associated to the fibration $X\to X_{hG} \to BG$, and may also be related to Bredon cohomology computations \cite{BG21}. In this paper, we carry out computations in the case of certain flag manifolds using an analogue of the novel technique of evaluating characteristic classes of wreath powers started in \cite{BBKV18}. 

Let $p$ be an odd prime. For the field $k=$ $\R$ or $\C$, consider the flag manifold 
$$F_n(k)= \{(V_1,\cdots, V_p)\mid V_i \subset k^{np},~ \dim(V_i)=n,~ V_i \perp V_j \mbox{ if } i \neq j \}.$$
Later we will denote $F_n(\C)$ by $F_n^U$, and $F_n(\R)$ by $F_n^{SO}$. The symmetric group $\Sigma_p$ acts on $F_n(k)$ by permuting the $V_i$, which is a free action. If we used $2$ instead of $p$, the flag manifold would be equivalent to the Grassmannian $G_n(k^{2n})$, and the $\Sigma_2$-action is the one which takes a subspace to its orthogonal complement. The index computations for this action are carried out in \cite{BBKV18}.  

We restrict the $\Sigma_p$ action on $F_n(k)$ to the cyclic group $C_p$, and we fix the coefficients for the cohomology as $\Z/p$. The flag manifold is the homogeneous space $F_n(\C)=U(pn)/U(n)^p$, and $F_n(\R) = O(pn)/O(n)^p$. As $p$ is odd, for a calculation involving cohomology with $\Z/p$ coefficients, we may work with $SO(n)$ instead of $O(n)$. Our first observation is that the cohomology of this flag manifold is concentrated in even degrees, and the spectral sequences associated to the fibrations
\[ F_n(\C)\to \Big(BU(n)\Big)^p \to BU(pn),~~F_n(\R)\to \Big(BSO(n)\Big)^p \to BSO(pn), \]
 lead us to a nice expression for it. Recall that the cohomology of $BC_p$ is given by 
\[H^\ast(BC_p;\Z/p) \cong \Z/p[u,v]/(u^2),~~|u|=1,~~ |v|=2. \]
In terms of this notation, we have the following result in the complex case. (see Theorems \ref{relprime}, \ref{Findex})
\begin{thma}
Let $n=p^aq$ with $p\nmid q$. Then, 
\[ \Index_{C_p}(F_n(\C))= (uv^{p^{a+1}-1}, v^{p^{a+1}}).\]
\end{thma}
The expression in the real case is quite analogous to the complex case with a slight difference when $q=1$. We have (see Theorems \ref{relprimeso}, \ref{indexfso})
\begin{thmb}
Let $n=p^aq$ with $p\nmid q$. Then, 
\[ \Index_{C_p}(F_n(\R))= \begin{cases} 
(v^{p^{a+1}-1}) & \mbox{ if } q=1 \\ 
(uv^{p^{a+1}-1}, v^{p^{a+1}}) & \mbox{ if } q>1. \end{cases}\]
\end{thmb}

We now describe geometric consequences associated with the index computation above. For a convex body $\CC$ inside $k^{pn}$, we consider the $p$-fold orthogonal shadow : $(p_{V_1}(\CC), \cdots p_{V_p}(\CC))$, where $V_i$ are mutually orthogonal subspaces of dimension $n$, and $p_{V_i}$ is the projection. We have the following conclusion about continuous functions defined on the space $\con(k^{pn})$ of convex bodies inside $k^{pn}$, the distance function being the Hausdorff metric. 

\begin{theorem} \label{func}
a) Let $n=p^aq$ such that $p\nmid q$,  and $r\leq 2(\frac{p^{a+1}-1}{p-1})$. Let $\alpha_1,\cdots, \alpha_r : \con (\C^{pn})\to \R$. For every proper convex body $C\subset \C^{pn}$, there exist $p$ mutually orthogonal $n$-dimensional subspaces $V_1,\cdots, V_p$ of $\C^{pn}$ such that
\[\alpha_i(p_{V_i}(C))=\cdots =\alpha_i(p_{V_p}(C))\quad \text{for all $1\leq i\leq r$.} \] 
b) Let $n$, $p$, $a$ and $q$ as above,  and  $r< 2(\frac{p^{a+1}-1}{p-1})$. Let $\alpha_1,\cdots, \alpha_r : \con (\R^{pn})\to \R$. For every proper convex body $C\subset \R^{pn}$, there exist $p$ mutually orthogonal $n$-dimensional subspaces $V_1,\cdots, V_p$ of $\R^{pn}$ such that
\[\alpha_i(p_{V_i}(C))=\cdots =\alpha_i(p_{V_p}(C))\quad \text{for all $1\leq i\leq r$.} \] 
\end{theorem}
\begin{proof}
We write the proof in the real case. The complex case is entirely similar. For a given convex $C\subset \R^{pn}$, $\alpha=(\alpha_1,\cdots,\alpha_r)$ may be used to construct a continuous function $F_C: F_n(\R)\to \R^{pr}$ by 
\[F_C(V_1,\cdots,V_p)=(\alpha(p_{V_1}(C)),\cdots,\alpha({p_{V_p}(C)}),\]
which is $C_p$-equivariant, where the right hand side is identified with $\rho^r$, a direct sum of $r$-copies of the regular representation.  If the hypothesis is not true, then the map $F_C$ avoids the diagonal, and the image of the projection $\tilde{F_C}$ onto the complementary subspace is non-zero at every point. The latter is a direct sum of $r$-copies of the reduced regular representation (denoted by $\overline{\rho}$). Thus we obtain a $C_p$-equivariant map $F_n(\R) \to S(\overline{\rho}^r)$, the unit sphere inside $\overline{\rho}^r$.  Now, $\Index_{C_p}(S(\overline{\rho}^r))=(v^{\frac{(p-1)r}{2}})$ \cite[Page 4]{BaKu21} and can not sit inside  $\Index_{C_p}(F_n(\R))$, which is either $(v^{p^{a+1}-1})$ or $(uv^{p^{a+1}-1},v^{p^{a+1}})$. This implies $\tilde{F_C}$ must be zero for some flag $(V_1,\cdots,V_p)$, a contradiction.
\end{proof}

There are interesting examples of continuous functions on the space of convex bodies which come from various measures. The following corollary is a direct consequence of the proof above, as in \cite[Corollary 1.4]{BBKV18}.
\begin{cor} \label{funccor}
With $n,p,q,a$ and $r$ as above, let $\alpha_1,\alpha_2,\cdots,\alpha_r: \con_n^{pn}\to \R$. Then, for every proper convex body $C\subset \R^{pn}$ containing the origin in its interior there exists $p$ mutually orthogonal $n$-dimensional subspaces $V_1,\cdots,V_p$ such that \[\alpha_i(C\cap V_1)=\cdots=\alpha_i(C\cap V_p)\] for all $1\leq i\leq r$. 
\end{cor} 

This result has multiple implications as was pointed out in \cite{BBKV18}. We prove the $p$-fold version of orthogonal transformations of inertia tensors in the following theorem. 
\begin{theorem}
Let $a\geq 0$ and $n=p^a$ or $2p^a$. Let $X\subset \R^{pn}$ be a finite set of points. There exist projections $P_1, P_2,\cdots, P_p: \R^{pn}\to \R^{pn}$ onto mutually orthogonal $n$-dimensional subspaces such that the $p$ inertia tensors 
\[I_{P_i}=\sum_{x\in X} P_i(x)\otimes P_i(x)\] for $i=\{1,\cdots, p\}$ 
are transformable from one to another by orthogonal transformations.
\end{theorem}
\begin{proof}
The configuration space of all projections $(P_1,\cdots, P_p)$ onto mutually orthogonal $n$-dimensional subspaces has an action of $C_p$ by cyclic permutations. This can be identified equivariantly with flag manifold $F_n(\R)$ by sending the $P_i$ to its image.  If we can show the characteristic polynomials $\det(I_{P_i}-\lambda I_{pn})$ are same for all the matrices of the bilinear forms defined by $I_{P_i}$, we are done. Note that as $\Ker(P_i)$ are $(p-1)n$-dimensional,  the characteristic polynomials have at least $(p-1)n$ zero roots. In all the characteristic polynomials we have $n$ non-zero coefficients in $\lambda^{pn-1},\cdots,\lambda^{pn-n}$. Let us denote them by $\alpha(P_i)=(\alpha_{pn-1}(P_i),\cdots,\alpha_{pn-n}(P_i))$ for $i\in\{1,\cdots,p\}$. Now consider the $C_p$-equivariant map 
\[F_I:F_n(\R)\to \R^{pn}, \mbox{ given by }
(P_1,\cdots,P_p)\mapsto (\alpha(P_1),\cdots, \alpha(P_p)), \]
 which intersects the diagonal by Theorem \ref{func}. Therefore, there exists a $(P_1,\cdots, P_p)$ such that $F_I(P_1,\cdots, P_p)=(\alpha(P_1),\cdots, \alpha(P_p))$ lies in the diagonal subspace $\Delta(\R^n)\subset \R^{pn}$ and thus, $\alpha(P_1)=\cdots=\alpha(P_p)$. This proves the theorem. 
\end{proof}

\begin{mysubsection}{Organization}
In \S \ref{cohflag}, we compute the cohomology of the equidimensional flag manifolds used in the document. In \S \ref{wrpowsp}, the cohomology  of the wreath power of spaces is noted down, and using this expression, we reduce the index computation to that of certain characteristic classes. In \S \ref{wreathcharclass}, the characteristic classes of the wreath power of vector bundles are computed with a view towards index calculations. In \S \ref{indcomp}, we complete the index calculations. 
\end{mysubsection}

\begin{notation}
Throughout the document, $p$ denotes an odd prime. We fix the notation $\sigma$ for a fixed generator of the cyclic group $C_p$ of order $p$. We use the  following notation
\begin{itemize}
\item $Gr_k(\C^n)$ denotes the complex grassmannian of $k$-planes in $\C^n$. As a homogeneous space, $Gr_k(\C^n)\cong U(n)/U(k)\times U(n-k)$. 
\item $Gr_k(\R^n)$ denotes the real grassmannian of $k$-planes in $\R^n$. As a homogeneous space, $Gr_k(\R^n)\cong O(n)/O(k)\times O(n-k)$.
\item $\tilde{Gr}_k(\C^n)$ denotes the oriented grassmannian of oriented $k$-planes in $\R^n$.  As a homogeneous space, $\tilde{Gr}_k(\R^n)\cong SO(n)/SO(k)\times SO(n-k)$.
\item $\gamma^n_U$ stands for the universal $n$-plane bundle over $BU(n)$, and $\gamma^n_{SO}$ stands for the universal $n$-plane bundle over $SO(n)$. 
\item $F_n^U$ denotes the flag manifold 
\begin{align*}
F_n^U &=\{V_1\subset V_2 \subset \cdots \subset V_p\subset \C^{pn}\mid\di(V_i)=ni\}\\
&= \{(W_1,\cdots, W_p),\, W_i \subset \C^{pn}  \mid  \di(W_i)=n;\,W_i \perp W_k \;\text{if}\,i\neq k\}.
\end{align*}
Observe that as a homogeneous space, $F_n^U \cong U(pn)/(U(n))^p$.
\item $F_n^{O}$ denotes the flag manifold 
\begin{align*}
F_n^{O} &=\{V_1\subset V_2 \subset \cdots \subset V_p\subset \R^{pn}\mid\di(V_i)=ni\}\\
&= \{(W_1,\cdots, W_p),\, W_i \subset \R^{pn}  \mid  \di(W_i)=n;\,W_i \perp W_k \;\text{if}\,i\neq k\}.
\end{align*}
Observe that as a homogeneous space, $F_n^{O} \cong O(pn)/(O(n))^p$.
\item We define $F_n^{SO} = SO(pn)/(SO(n))^p$.  For the purposes of this paper, it suffices to work with $F_n^{SO}$ instead of $F_n^O$. 
\item For $r\leq p$, $F_{n,r}^U$ denotes the flag manifold 
\[
F_{n,r}^U = \{(W_1,\cdots, W_r),\, W_i \subset \C^{pn}  \mid  \di(W_i)=n;\,W_i \perp W_k \;\text{if}\,i\neq k\}. \]
As a homogeneous space, $F_{n,r}^U \cong U(pn)/(U(n)^r \times U((p-r)n)$. 
\item For $r\leq p$, $F_{n,r}^{SO} = SO(pn)/ (SO(n)^r \times SO((p-r)n)$. 
\end{itemize}
 In the case where the coefficient group in $H^\ast(X)$ is not specified, it is assumed to be $\Z/p$. For a formal sum of cohomology classes $\phi$ $\in H^\ast X$, we denote the degree $2k$ part of $\phi$ by $[\phi]_k$.
\end{notation}

\begin{ack} 
The first author would like to thank Surojit Ghosh for certain helpful conversations. The research of the second author was supported by CSIR-SRF 09/934(0008)/2017-EMR1.
\end{ack}
\section{Cohomology of flag manifolds}\label{cohflag}
The main objective of this section is to compute the cohomology of $F_j^U$ and $F_j^{SO}$. We start with the unitary case. Along the way we also compute the cohomology of $F_{j,r}^U$ which is defined as 
\[
F_{j,r}^U = \{(W_1,\cdots, W_r),\, W_i \subset \C^{pj} \, \mid \, \di(W_i)=j;\,W_i \perp W_k \;\text{if}\,i\neq k\}.\]
One directly observes that $F_{j,r}^U$ is homeomorphic to  $U(jp)/{{U(j)}^r \times U(pj-rj)}$. Forgetting the last subspace $W_r$ gives the fibration
\begin{myeq}\label{F1}
Gr_j(\C^{(p-r+1)j})\xrightarrow{i} F_{j,r}^{U}\to F_{j,{r-1}}^U,
\end{myeq}
where $Gr_k(\C^n)$ is the grassmannian of $k$-planes in $\C^n$. Recall that the cohomology of $Gr_k(\C^n)$ is computed via the fibration 
\[
Gr_k(\C^n)\to BU(k)\times BU(n-k) \to BU(n).
\]
We denote the graded algebra $H^\ast(Gr_k(\C^n))$ by $H_{n,k}$. Recall that 
$$H^\ast(BU(n))\cong \Z/p[c_1,\cdots,c_n],$$
 where $c_i$ are the$\pmod p$ reductions of the integral Chern classes. That is, $c_i$ is the$\pmod p$ reduction of the $i^{th}$ Chern class of $\gamma^n_U$, the universal $n$-plane bundle over $BU(n)$. The algebra $H_{n,k}$ has the form 
\begin{myeq}\label{Hdefn}
H_{n,k}= \Z/p[c_1,\cdots,c_k,c_1',\cdots,c_{n-k}']/(\hat{c}_1,\cdots, \hat{c}_{n})=\Z/p[c_1,\cdots,c_k]/(\tilde{c}_{n-k+1},\cdots, \tilde{c}_{n}).
\end{myeq}
The elements $c_i$ and $c_i'$ are pull-backs of the Chern classes over $BU(k)$ and $BU(n-k)$ respectively. Here, $\hat{c_i}$ are defined by the equation 
\[
1+\hat{c}_1+\cdots +\hat{c}_n = (1+c_1+\cdots+c_k)(1+c_1'+\cdots +c_{n-k}').
\]
In this equation, for $i\leq n-k$, $\hat{c}_i$ has the form $c_i+c_i'$ plus lower order terms in the $c_i$ and $c_i'$. Therefore, it serves as a means of expressing  the $c_i'$ inductively in terms of the $c_i$ in the algebra $H_{n,k}$. Finally for $i>n-k$, we may incorporate this formula of the $c_i'$ in terms of the $c_i$ into $\hat{c}_i$ to obtain $\tilde{c}_i$ in the expression above. We now proceed towards the computation for $F_j^U$.  
\begin{prop}\label{P1}
The cohomology groups of $F_{j,r}^U$ are concentrated in even degrees.
\end{prop}
\begin{proof}
Using \eqref{F1}, we  proceed by induction on $r$. For $r=1$, $F_{j,1}^U$ is homeomorphic to $Gr_j(\C^{jp})$, which by \eqref{Hdefn} is concentrated only in even degrees. 
Now assume that $H^*(F_{j,{r-1}}^U)$ is concentrated in even degrees. Then, in the Serre spectral sequence for \eqref{F1}, the $E_2$-page is concentrated in even bidegrees. Therefore, all the differentials are forced to be $0$ for degree reasons, and the result follows.
\end{proof}

An explicit formula for the cohomology ring of $F_j^U$ is now derived. It involves computations with the fibration
 \begin{myeq}\label{flagr}
F_{j,r}^U\to B(U(j)^r\times U(p-r)j) \xrightarrow{p} BU(jp).
\end{myeq}
 The map 
\[
BU(j)^r \times BU((p-r)j) \to BU(pj) 
\]
classifies the bundle $\oplus_{i=1}^r \pi_i^\ast\gamma^j_U \oplus \pi_0^\ast \gamma^{(p-r)j}_U$, where 
\[
\pi_i : BU(j)^r \times BU((p-r)j) \to BU(j) ~\mbox{ (}i^{th}\mbox{ factor)},\,
\]
\[
 \pi_0:BU(j)^r \times BU((p-r)j) \to BU((p-r)j).
\]   
We now have the formula $x_i=p^\ast(c_i)$ for $1\leq i \leq jp$, where the $x_i$ are defined by 
\[
1+x_1+\cdots + x_{pj}= \Big[\prod_{i=1}^r (1+c_{1,i}+\cdots + c_{j,i})\Big] (1+c_{1,0}+\cdots + c_{(p-r)j,0}). 
\]
Here the notation is defined as 
\[
c_{i,l}=\pi_l^\ast(c_i)~\mbox{ for } 1\leq i \leq j , ~ c_{i,0}=\pi_0^\ast(c_i) ~\mbox{ for } 1\leq i \leq (p-r)j.
\] 
In the following theorem, the notation $c_{i,l}$ also refers to the image in the cohomology of the flag manifold $F_{j,r}^U$. 
\begin{theorem}\label{T2}
In terms of the notations above,
$$H^*(F_{j,r}^U)=\Z/p[c_{i,l}]/( x_1,\cdots,x_{jp}).$$
\end{theorem}
\begin{proof}
The Serre spectral sequence for \eqref{flagr} collapses at the $E_2$-page as both $H^\ast(F_{j,r}^U)$ (by Proposition \ref{P1}) and $H^\ast(BU(pj))$ are concentrated in even degrees. From the convergence of the spectral sequence, we see that the cohomology of the fibre $F_j^U$ is obtained from the cohomology of the total space by quotienting out the ideal generated by pulling back the positive degree elements from the base. The result is now immediate. 
%
\end{proof}

In the case $G=SO$, the arguments are slightly more delicate as it is no longer true that the corresponding grassmannian has a CW complex structure with only even degree cells. Recall the cohomology of $BSO(n)$ and $SO(n)$, \cite{Bor53} 
\[H^*(BSO(n)) \cong
\begin{cases}
\Z/p[p_1,p_2,\cdots, p_{\frac{n-1}{2}}]\quad &\text{if $n$ odd},\\
\Z/p[p_1,p_2,\cdots, p_{\frac{n}{2}-1},e_n]\quad &\text{if $n$ even}
\end{cases}\]
where $p_i$ are the Pontrjagin classes of the universal bundle with $\deg(p_i)=4i$, and $e_n$ is the Euler class with $\deg(e_n)=n$ which is non-zero only in the even case; and
\[H^*(SO(n);\Z/p) \cong
\begin{cases}
\Lambda_{\Z/p}[y_1,y_2,\cdots, y_{\frac{n-1}{2}}]\quad &\text{if $n$ odd},\\
\Lambda_{\Z/p}[y_1,y_2,\cdots, y_{\frac{n}{2}-1},\sigma_{n-1}]\quad &\text{if $n$ even},
\end{cases}\] 
where $\deg(y_i)=4i-1$ and $\deg(\sigma_{n-1})=n-1$.
In the following proposition we describe the cohomology of the oriented grassmannian $\tilde{Gr}_{n,k} = SO(n)/SO(k)\times SO(n-k)$. We use the notation $H'_{n,k}$ to denote the graded algebra which is abstractly isomorphic to $H_{n,k}$ via a degree doubling isomorphism. That is, 
\[
H_{n,k}'^{(2s)} = H_{n,k}^{(s)}, \mbox{ and } H_{n,k}'^{(2s-1)} = 0. 
\]
We suggestively write $p_i$ and $p_i'$ instead of $c_i$ and $c_i'$ in the algebra $H_{n,k}'$ to denote the corresponding image via the degree doubling isomorphism. 
\begin{prop}\label{P2}
We have the following formula for the cohomology of $\tilde{Gr}_j(\R^n)$ if $j>1$
\[
H^*(\tilde{Gr}_j(\R^n)) \cong \begin{cases} 
H'_{\frac{n-1}{2},\frac{j-1}{2}}[e_{n-j}]/(e_{n-j}^2=p'_{\frac{n-j}{2}}) &\mbox{if } n,j \mbox{ are odd}  \\
H'_{\frac{n-1}{2},\frac{j}{2}}[e_{j}]/(e_{j}^2=p_{\frac{j}{2}}) &\mbox{if } n \mbox{ is odd and } j \mbox{ is even}\\
H'_{\frac{n}{2},\frac{j}{2}}[e_j,e_{n-j}]/( e_{n-j}^2=p'_{\frac{n-j}{2}}, e_j^2=p_{\frac{j}{2}},e_je_{n-j}) &\mbox{if } n,j \mbox{ are even}\\
H'_{\frac{n}{2},\frac{j-1}{2}}[\sigma_{n-1}]/( \sigma_{n-1}^2) &\mbox{if } n \mbox{ is even and } j \mbox{ is odd}.
\end{cases}
\]
%
%
%
%
%

\end{prop}
\begin{proof}
The proof of this proposition follows from the Serre spectral sequence associated to the fibration 
\begin{myeq}\label{F2}
 \tilde{Gr}_j(\R^{n}) \to BSO(j)\times BSO(n-j)\to BSO(n). 
\end{myeq}
This may be computed via the techniques of \cite{Bor53}. For example all the cases other than the last one follow from \cite[Theorem 26.1]{Bor53}. In the last case, one observes that in the spectral sequence for \eqref{F2}, the Euler class $e_n$ pulls back to $0$ in $BSO(j)\times BSO(n-j)$. The class $\sigma_{n-1}$ transgresses to this element. The rest of the spectral sequence works analogously as the spectral sequence for
\[ Gr_{\frac{j-1}{2}}(\C^{\frac{n}{2}}) \to BU(\frac{j-1}{2})\times BU(\frac{n-j+1}{2}) \to BU(\frac{n}{2}). \]
\end{proof}

 Proposition \ref{P2} directly yields the following corollary 
\begin{cor}\label{C1}
The cohomology of $\tilde{Gr}_j(\R^n)$ for $j>1$ is concentrated in even degrees unless $n$ is even and $j$ is odd.
\end{cor}
Now we turn our attention to $F_j^{SO}$ and consider the real flag manifold $F_{j,r}^{SO}$ defined as the homogeneous space $F_{j,r}^{SO}=SO(jp)/{SO(j)^r\times SO((p-j)r)}$. 
As in the unitary case, we consider the fibration 
 \begin{myeq}\label{flagsor}
F_{j,r}^{SO}\to B(SO(j)^r\times SO(p-r)j) \xrightarrow{p} BSO(jp).
\end{myeq}
The map 
\[
BSO(j)^r \times BSO((p-r)j) \to BSO(pj) 
\]
classifies the bundle $\oplus_{i=1}^r \pi_i^\ast\gamma^j_{SO} \oplus \pi_0^\ast \gamma^{(p-r)j}_{SO}$, where 
\[
\pi_i : BSO(j)^r \times BSO((p-r)j) \to BSO(j) ~\mbox{ (}i^{th}\mbox{ factor)},\,
\]
\[
 \pi_0:BSO(j)^r \times BSO((p-r)j) \to BSO((p-r)j).
\]   
We now have the formula $x_i=p^\ast(p_i)$ for $1\leq i \leq \floor{jp/2}$, where the $x_i$ are defined by 
\[
1+x_1+\cdots + x_{\floor{pj/2}}= [\prod_{i=1}^r (1+p_{1,i}+\cdots + p_{\floor{j/2},i})] (1+p_{1,0}+\cdots + p_{\floor{(p-r)j/2},0}). 
\]
Here the notation is defined as 
\[
p_{i,l}=\pi_l^\ast(p_i)~\mbox{ for } 1\leq i \leq \floor{j/2} , ~ p_{i,0}=\pi_0^\ast(p_i) ~\mbox{ for } 1\leq i \leq \floor{(p-r)j/2}.
\] 
If $j$ is even, we also have Euler classes $e_{j,l}=\pi_l^\ast(e_j)$ for $1\leq l \leq r$. In this case $(p-r)j$ is also even, and so we also have the Euler class $e_{(p-r)j,0}=\pi_0^\ast e_{(p-r)j}$. We have the formula,  
\[ p^\ast(e_{jp})=\prod_{l=1}^r e_{j,l} \cdot e_{(p-r)j,0}.\]
If $j$ is odd, and $r$ is also odd, then $(p-r)j$ is even, and we have the element $e_{(p-r)j,0}=\pi_0^\ast e_{(p-r)j}$ in the cohomology ring of $BSO(j)^r \times BSO((p-r)j)$. As in Theorem \ref{T2}, we use the same notation to denote the image in the cohomology of the flag manifold $F_{j,r}^{SO}$. 
\begin{theorem}\label{TSO}
The cohomology of $F_{j,r}^{SO}$ for $j>1$ is concentrated in even degrees, and we have in terms of the notation above, 
\[
H^*(F_{j,r}^{SO}) \cong 
\begin{cases}
 \frac{\Z/p[p_{i,l}]}{( x_1,\cdots, x_{\floor{pj/2}})} &\text{if $r$ is even, $j$ is odd,}\\ \\
 \frac{\Z/p[p_{i,l},e_{(p-r)j,0}]}{( x_1,\cdots, x_{\floor{pj/2}},e_{(p-r)j,0}^2-p_{\frac{(p-r)j}{2},0})} &\text{if $r$ is odd, $j$ is odd,}\\ \\
\frac{ \Z/p[p_{i,l},e_{j,l},e_{(p-r)j,0}]}{( x_1,\cdots, x_{\floor{pj/2}}, p^\ast e_{jp},e_{j,l}^2-p_{j,l}, e_{(p-r+1)j,0}^2-p_{\frac{(p-r)j}{2},0})} &\text{if $j$ is even.}
\end{cases}
\]
\end{theorem}
\begin{proof}
We  proceed by induction on $r$. For $r=1$ this is just the oriented real Grassmannian manifold $\tilde{Gr}_j(\R^{jp})$, and the result is implied by Proposition \ref{P2} and Corollary \ref{C1}. In the induction step, we compute via the Serre spectral sequence associated to the fibration 
\begin{myeq}\label{F4}
\tilde{Gr}_j(\R^{(p-r+1)j})\to F_{j,r}^{SO}\to F_{j,r-1}^{SO},
\end{myeq} 
assuming the expression for $F_{j,r-1}^{SO}$. We have 2 cases to consider. \\

\noindent
\textbf{Case I:} Either $j$ is even or $r$ is odd.\\
The assumptions on $r$ and $j$ imply using Corollary \ref{C1} that $H^\ast(\tilde{Gr}_j(\R^{(p-r+1)j})$ is concentrated in even degrees. This implies via induction that the cohomology of $F_{j,r}^{SO}$ is concentrated in even degrees. Now we compute the Serre spectral sequence for the fibration \eqref{flagsor}, which degenerates at the $E_2$-page due to degree reasons. This implies that $H^\ast(F_{j,r}^{SO})$ is the quotient of $H^\ast(BSO(j)^r\times BSO((p-r)j))$ by the ideal generated by $p^\ast$ applied to the positive degree classes of $H^\ast(BSO(pj))$. The result now follows. \\

\noindent
\textbf{Case II:} $j$ is odd and $r$ is even.\\
In this case there is an odd degree class $\sigma_{(p-r+1)j-1}$ in $H^*(\tilde{Gr}_j(\R^{(p-r+1)j}))$. We show that this class does not survive the spectral sequence for \eqref{F4}. Inductively, we have
 \[H^*(F_{j,r-1}^{SO})\cong  \frac{\Z/p[p_{i,l},e_{(p-r+1)j,0}]}{( x_1,\cdots, x_{\floor{pj/2}},e_{(p-r+1)j,0}^2-p_{\frac{(p-r)j}{2},0})} .\]
We prove that $\sigma_{(p-r+1)j-1}$ transgresses to the class $e_{(p-r+1)j,0}$. This implies that the odd degree classes in the spectral sequence for \eqref{F4} support a differential, and do not survive to $E_\infty$. The rest follows as in the previous case via the fibration \eqref{flagsor}. 

Consider the vector bundle over $F_{j,r-1}^{SO}$ classified by the map 
$$\xi_{(p-r+1)j}: F_{j,r-1}^{SO}\to BSO((p-r+1)j)$$
 whose fibres are $(\oplus_1^{r-1}W_i)^{\perp}$. The quotient map 
$$q: SO((p-r+1)j)\to \tilde{Gr}_j(\R^{(p-r+1)j})$$ 
induces the following map between the fibrations.
\[
\xymatrix{
SO((p-r+1)j)\ar[rr]\ar[d] & & \tilde{Gr}_j(\R^{(p-r+1)j}) \ar[d]\\
Fr(\xi_{(p-r+1)j}) \ar[rr] \ar[d] & & \tilde{Gr}_j(\xi_{(p-r+1)j})\simeq F_{j,r}^{SO} \ar[d]\\
F_{j,r-1}^{SO} \ar@{=}[rr] & & F_{j,r-1}^{SO}
}
\]
Here $Fr(\xi_{(p-r+1)j})$ denotes the oriented frame bundle of $\xi_{(p-r+1)j}$. The differentials for the Serre spectral sequence of the right hand fibration are computed via the left hand fibration. The differentials for left hand fibration are computed via pulling back the universal bundle by the following commutative diagram.
\[
\xymatrix{
SO((p-r+1)j)\ar@{=}[rr]\ar[d] & & SO((p-r+1)j)\ar[d]\\
Fr^*(\xi_{(p-r+1)j}) \ar[rr] \ar[d] & & ESO((p-r+1)j) \ar[d]\\
F_{j,r-1}^{SO} \ar[rr] & & BSO((p-r+1)j)
}
\]
We now easily compute that 
 \[d_i(\sigma_{(p-r+1)j-1})=0 ~~\mbox{ if } i< (p-r+1)j, \mbox{ and}\]
\[d_{(p-r+1)j}(\sigma_{(p-r+1)j-1})=e_{(p-r+1)j,0}.\]
%
%
%
%
\end{proof}

\section{Cohomology of the wreath power of spaces} \label{wrpowsp}
In this section, we reduce the computation of the index of the flag manifold to the calculation of characteristic classes via the wreath power construction. These characteristic classes are computed in \S \ref{wreathcharclass}. The $p$-fold wreath power construction is an odd primary analogue of the wreath squares of \cite{BBKV18}.

For a CW-complex $X$, the cyclic group $C_p$ of order $p$ acts on $X^p$ by 
\[\sigma\cdot(x_1,\cdots,x_p)=(x_p,x_1,\cdots,x_{p-1})\] 
where $\sigma$ is the generator of $C_p$. The $p$-th wreath power of $X$ is the Borel construction on $X^p$ denoted by $X^p_{hC_p} $, and we have the usual fibre bundle 
\begin{myeq}\label{F6}
X^p\to X^p_{hC_p}\to BC_p.
\end{myeq}
Recall that the cohomology of $BC_p$ has the form 
\begin{myeq}\label{cohcp}
H^\ast(BC_p)\cong \Z/p[u,v]/(u^2), ~~ |u|=1,	~ |v|=2.
\end{myeq}
The $E_2$-term of the Serre spectral sequence of the fibration \eqref{F6} is given by 
\[E_{2}^{i,j}:=H^i(C_p;\mathcal{H}^j(X^p)).\]
Here, the local coefficient system is determined by the action of $\pi_1(BC_p)=C_p = \langle \sigma \rangle$ on $H^*(X^p)\cong H^*(X)\otimes\cdots\otimes H^*(X)$, which is given by 
\[\sigma\cdot(x_1\otimes\cdots\otimes x_p)=(x_p\otimes x_1\otimes \cdots\otimes x_{p-1}).\] 
Finally the $E_2$-term of the Serre spectral sequence associated to the fibration \eqref{F6} is given by \cite[Corollary IV.1.6]{AM04}
\begin{myeq}\label{E1}
E_2^{i,j}=
\begin{cases}
{H^j(X^{p})}^{C_p}& \quad \text{if $i=0$},\\
H^{r}(X)& \quad \text{if $i>0$ and $j=pr$},\\
0 & \quad \text{otherwise}.
\end{cases}
\end{myeq}
\begin{prop}\cite[Theorem IV.1.7]{AM04}\label{P3}
 The Serre spectral sequence for the fibration \eqref{F6} degenerates at $E_2$-page.
\end{prop}
We introduce the following maps in order to give a description of the $E_2$-page $=E_{\infty}$-page.
\begin{myeq}\label{E2}
\begin{aligned}
&P:H^r(X)\to H^{pr}(X^p)\cong E_{\infty}^{0,pr} ,  \mbox{ given by } x\mapsto x^{\otimes p}, & \\
&I:H^j(X^p)\to {H^j(X^p)}^{C_p},  \mbox{ given by } x_1\otimes\cdots\otimes x_p \mapsto \sum_{g\in C_p}g \cdot (x_1\otimes\cdots\otimes x_p). & 
\end{aligned}
\end{myeq}
Note that the map $P$ is multiplicative, and $I$ is additive. Now we may write the $E_{\infty}$-page with the help of the maps $P$ and $I$, and one has the following identifications
\begin{myeq}\label{Einfty}
\begin{aligned}
& E_{\infty}^{i,0}\cong H^i(BC_p)\cong \Z/p[u,v]/( u^2),\\
& E_{\infty}^{0,j}\cong {H^j(X^p)}^{C_p},\\
& E_{\infty}^{0,pr} \cong P(H^{r}(X))\oplus I(H^{pr}(X^p)),\\
& E_{\infty}^{i,pr} \cong P(H^{r}(X)) \otimes H^i(BC_p),\\
& I(H^j(X^p))\cdot u=0, \, I(H^j(X^p))\cdot v=0. 
\end{aligned} 
\end{myeq}
 These relations describe the complete ring structure of $E_{\infty}$-page. There are no  multiplicative extension problems from  \cite[Remark after Theorem 2.1]{Le97}. We summarize this in the following result. 
\begin{thm} \label{cohwrpow}
The cohomology of $X^p_{hC_p}$ is generated over $H^\ast(BC_p)$ by $P(H^r(X))$ and $I(H^r(X^p))$ modulo of the ideal generated by terms $I(y)\cdot v$, $I(y)\cdot u$ for $y \in H^r(X^p)$. 
\end{thm}
The following proposition follows from the description of the $E_\infty$-page above. 
\begin{prop}\label{P7}
If $f:B_1\to B_2$ induces an injective map $f^*:H^*(B_2)\to H^*(B_1)$, then 
\[(f^p_{hC_p})^*:H^*({B_2}^p_{hC_p})\to H^*({B_1}^p_{hC_p})\] 
is injective. 
\end{prop}
\begin{proof}
We have the following commutative diagram of fibrations
\[
\xymatrix{
B_1^p \ar[d] \ar[rr] & & B_2^p \ar[d] \\
{B_1}^p_{hC_p}  \ar[rr]^{f^p_{h C_p}}\ar[d] & & {B_2}^p_{h C_p}\ar[d]\\
BC_p\ar@{=}[rr] & & BC_p.
}
\] 
The description of the $E_2$-page of the Serre spectral sequence of both the fibrations are given by \eqref{E1}. The map between the two fibrations induces a map between the corresponding spectral sequences, and from Proposition \ref{P3}, it follows that $E_2=E_{\infty}$, and the spectral sequence converges to the cohomology of $p$-th wreath power of corresponding $B_i$. The description of the $E_2$-page shows that the map is injective on the $E_2$-page, and as the spectral sequence degenerates at $E_2$, this is also injective on the $E_\infty$-page.  
\end{proof}

\begin{mysubsection}{The wreath power of $BU(n)$ and $BSO(n)$}
Let $G$ denote one of $U$ or $SO$. Using the action of  $C_p$  on $G(n)^p$ via the cyclic permutation of coordinates,  we define 
\[W_n^G :=G(n)^p\rtimes C_p,\]
 the semidirect product induced by the action of $C_p$ on $G(n)^p$. That is, the notation $W_n^G$ refers to  $W_n^U= U(n)^p \rtimes C_p$ if $G=U$, and $W_n^{SO}= SO(n)^p \rtimes C_p$ if $G=SO$. 
This gives us an exact sequence of groups 
\begin{myeq}\label{F7}
0\to G(n)^p \to W_n^G \to C_p\to 0.
\end{myeq} 
The exact sequence of groups induces the fiber bundle 
\begin{myeq}\label{F8}
B(G(n)^p)\to BW_n^G \xrightarrow{\pi} BC_p.    
\end{myeq}
 The Serre spectral sequence associated to \eqref{F8} is the Lyndon-Hochschild-Serre spectral sequence associated to \eqref{F7}. We also observe that \eqref{F8} identifies $BW_n^G$ as the wreath power of $BG(n)$. The associated spectral sequence has all differentials $0$ from the second page onwards (Proposition \ref{P3}). Therefore,
\begin{prop}\label{P4}
$\pi^*:H^*(BC_p;\Z/p)\to H^*(BW_n^G;\Z/p)$ is injective.
\end{prop}

\vspace{0.5cm}
\end{mysubsection}

\begin{mysubsection}{The index of the flag manifold} 
We  now describe an useful reduction for the index of flag manifold. The techniques are analogous to \cite{BBKV18}. We again let $G$ be one of $U$ or $SO$, and the cyclic group $C_p$ acts on the complex flag manifold $F_n^G$ by  cyclically permuting the orthogonal subspaces 
\[\sigma\cdot (W_1,\cdots, W_p)= (W_p,W_1,\cdots, W_{p-1}). \] 
%
The group $W_n^G = G(n)^p \rtimes C_p$ embeds in $G(np)$ via the maps 
\[
(A_1,\cdots,A_p)\in G(n)^p  \mapsto \begin{pmatrix}
A_1 & 0 &\cdots & 0\\
0 & A_2 &\cdots & 0\\
\cdot &\cdot &\cdot &\cdot\\
0 & 0 &\cdots & A_p
\end{pmatrix},
\sigma \mapsto \begin{pmatrix}
0 & 0 &\cdots & I\\
0 & \cdots &I & 0\\
\cdot &\cdot &\cdot &\cdot\\
I & 0 &\cdots & 0
\end{pmatrix}.\]
This leads to the homeomorphism 
 \[
F_n^G/C_p\cong \Big[G(pn)/G(n)^p\Big]/C_p \cong G(pn)/(G(n)^p\rtimes C_p) \cong G(pn)/W_n^G.\]
 Since the action of $W$ on $U(jp)$ is free we can further assert that
\begin{myeq}\label{heF}
{F_n^G}_{hC_p}=EC_p\times_{C_p}F_n^G \simeq F_n^G/C_p \cong G(pn)/W_n^G\simeq  EG(pn)\times_{W_n^G}G(pn).
\end{myeq} 
Thus we have the following diagram whose columns are fibrations
\begin{myeq}\label{CD1}
	\xymatrix{
    F_n^G  \ar[d] & G(pn) \ar@{=}[rr] \ar[l] \ar[d]  && G(pn) \ar[d] \\
	(F_n^G)_{hC_p}\ar[d]&EG(pn)\times_{W_n^G} G(pn)\ar[rr]\ar[d]_{p_1}\ar[l]^-{\simeq} &  & EG(pn)\ar[d] \\
	BC_p& BW_n^G \ar[rr]^{i_B}\ar[l]_{\pi} &  & BG(pn).
	}
\end{myeq}
The homotopy equivalence \eqref{heF} induces the left part of the commutative diagram. The right part of the diagram is induced by $W_n^G \xhookrightarrow{} G(pn)$, which gives the map $i_B$ on classifying spaces, and then $p_1$ is the pullback of the universal $G(pn)$-fibration to $BW_n^G$ via $i_B$. Recall that the index of $F_n^G$ is given by 
\[ 
\Index_{C_p}(F_n^G)= \Ker(H^\ast(BC_p) \to H^\ast({F_n^G}_{hC_p}).
\]
Applying Proposition \ref{P4}, we see that $\pi^\ast$ is injective, and via the diagram \eqref{CD1}, we have,
\[ 
\Index_{C_p}(F_n^G)= \Ker(p_1^\ast \circ \pi^\ast) = (\pi^\ast)^{-1}\Ker(p_1^\ast).
\]
In the right fibration of \eqref{CD1}, the differentials in the Serre spectral sequence are computed using characteristic classes. The cohomology of $G(pn)$ is an exterior algebra (only additively if $p=2$ and $G=SO$) and the generators transgress to appropriate characteristic classes \cite{Bor53}. This implies that the  kernel of $p_1^\ast$ is the ideal generated by the pullbacks of the characteristic classes under $i_B^\ast$. Appropriate formulae for these are computed in \S \ref{wreathcharclass}. 

Inspecting the cohomology ring $\Z/p[u,v]/(u^2)$ of $BC_p$ in which $\Index_{C_p}(F_n^G)$ is an ideal, we see that this may have the form $(v^l)$ or  $(uv^l, v^r)$. In the latter case, we also notice that as the B\"{o}ckstein sends $u$ to $v$, the ideal must be of the form $(uv^{l-1},v^l)$. Therefore, the computation of the index reduces to the following.  

%
%
\begin{prop}\label{indred}
The Fadell-Husseini index $\Index_{C_p}(F_n^G)$ equals either $(uv^{l-1},v^l)$ or $(v^l)$, where $\pi^\ast (uv^{l-1})$ or $\pi^\ast(v^l)$ is the lowest degree non-zero element of $\Im(\pi^*)\cap \Ker(p_1^\ast)$. 
\end{prop}
\end{mysubsection}

\section{The wreath power of vector bundles}\label{wreathcharclass}
In this section, we construct the wreath power of a vector bundle, and compute it's characteristic classes in terms of those of the original vector bundle. This is the $p$-power analogue of the wreath square defined in \cite{BBKV18} for odd primes $p$.  We compute the Chern classes in the case of complex bundles, and the Pontrjagin classes in the case of real bundles.  The expressions obtained are used in the computation of the index of $F^G_n$ for $G=U$ and $G=SO$ in the following section.

\begin{mysubsection}{The $p$-fold wreath power of a vector bundle}
Let $\xi:E\to B$ be a $n$-dimensional real or complex vector bundle. The $p$-fold product  bundle $\xi\times\cdots\times \xi$ with total space $E^p$ and base $B^p$ is equipped with a $C_p$-action by cyclically permuting the factors via 
\[\sigma\cdot(y_1,\cdots,y_p)=(y_p,y_1,\cdots,y_{p-1})\] 
where $\sigma$ is the generator of $C_p$. The pullback bundle induced by the projection 
 \[
p_1: B^p\times EC_p\to B^p
\] 
has total space $E^p \times EC_p$. Note that the diagonal $C_p$-action on the space $B^p\times EC_p$ is free, it induces a bundle on the quotient spaces
\[ \xi \wr C_p : E^p_{hC_p} \to B^p_{hC_p}.\]
The bundle $\xi\wr C_p$ is of dimension $pn$, and is called $p$-fold wreath power of $\xi$. We note
\begin{prop}\label{P5} The wreath power construction preserves direct sums. That is, for two bundles $\xi_1$ and $\xi_2$,
\[(\xi_1\oplus\xi_2)\wr C_p\cong (\xi_1\wr C_p)\oplus(\xi_2\wr C_p).\]
\end{prop}

\begin{proof}
The obvious choice of isomorphism 
\[a:(F_{\xi_1}\oplus F_{\xi_2})\times\cdots\times (F_{\xi_1}\oplus F_{\xi_2})\to (F_{\xi_1}\times\cdots\times F_{\xi_1})\oplus (F_{\xi_2}\times\cdots\times F_{\xi_2}) \]
induces isomorphism between the fibers of $p$-th wreath power of Whitney sum of bundles and Whitney sum of $p$-th wreath power of bundles, and hence the proposition follows. 
\end{proof}
One also has the following naturality result for the wreath power construction. 
\begin{prop}\label{P6}
The wreath power of vector bundles behaves naturally with respect to pull-backs of vector bundles. In other words,
 \[(f\wr C_p)^*(\xi\wr C_p)=f^*\xi\wr C_p \] 
where $\xi: E\to B$ is a vector bundle, and $f:X\to B$ is a continuous map. 
\end{prop}
\begin{proof}
The proof follows verbatim from \cite[\S 3.3]{BBKV18} replacing the wreath square by the $p$-fold wreath power.
\end{proof}

\end{mysubsection}

\begin{mysubsection}{The Chern class of a wreath power}
We now compute the ($\pmod p$ reduction of the) Chern classes of $\xi\wr C_p$ for a complex vector bundle $\xi$. These lie in the cohomology groups $H^\ast(X^p_{hC_p})$, whose expression we note from Theorem \ref{cohwrpow}. We introduce the following notation using the definitions in \eqref{Einfty}. 
\begin{notation}\label{z}
Let $\phi\in H^\ast(X)$ be a sum of homogeneous classes $\phi=\sum_{i=1}^r \phi_i$. We define $z(\phi)\in H^\ast(X^p_{hC_p})$ by the formula 
\[ z(\phi) = P(\phi)-\sum_{i=1}^r P(\phi_i).\]
Note that $z(\phi)$ lies in the image of $I$. We also write $z_k(\phi)$ to denote the degree $k$ homogeneous part of $z(\phi)$. Observe that if $\phi=c_0+c_1$ with $|c_0|=0$ and $|c_1|=1$, then 
\[
z(\phi)=  \sum_{\substack{\small q_i\in \{0,1\} \\(q_1,\cdots,q_p)\neq (0,\cdots,0) \\ (q_1,\cdots,q_p)\neq (1,\cdots,1)}}c_{q_1}(\xi)\otimes\cdots\otimes c_{q_p}(\xi) . \]
If we further put $c_0=1$, the homogeneous parts of $z(1+c_1)$ are just the elementary symmetric polynomials on the terms $1\otimes \cdots \otimes c_1 \otimes \cdots \otimes 1$. 
\end{notation}

 We start by computing the Chern classes of the wreath power of complex line bundles.  Note that for a complex bundle $\xi: E\to B$ and a map $f:X\to B$ , Proposition \ref{P6}  implies 
\[(f\wr C_p)^*(c(\xi\wr C_p))=c(f^*(\xi)\wr C_p),\] 
where $c(\xi\wr C_p)$ denotes the total Chern class of bundle $\xi\wr C_p$. Recall that the cohomology of $BC_p$ is given by \eqref{cohcp}
\[H^*(BC_p)=\Z/p[u,v]/( u^2) \]
 where $\deg(u)=1$ and $\deg(v)=2$.
\begin{prop}\label{P8}
Let $\xi$ be a $1$-dimensional complex vector bundle over a CW-complex $B$. Then
\begin{align*}
c(\xi\wr C_p)&=(1+v^{p-1})+P(c_1(\xi))+\sum_{\substack{\small q_i\in \{0,1\} \\(q_1,\cdots,q_p)\neq (0,\cdots,0) \\ (q_1,\cdots,q_p)\neq (1,\cdots,1)} }c_{q_1}(\xi)\otimes\cdots\otimes c_{q_p}(\xi) \\ 
 &= (1+v^{p-1})+P(c_1(\xi))+z(c(\xi)).
\end{align*} 
where $P$ is as defined in \eqref{E2}.
\end{prop}
\begin{proof}
We proceed as in \cite[Proposition 3.5]{BBKV18}. Consider the map between bundles
\[
\xymatrix{
E^p_{hC_p} \ar@{=}[r] & E(p_1^*(\xi^p))/C_p\ar[d] & & E(p_1^*(\xi^p))\ar[ll]\ar[d]\ar[rr] & & E^p\ar[d]^{\xi^p}\\
B^p_{hC_p} \ar@{=}[r] & B^p\times_{C_p} EC_p & & B^p\times EC_p\ar[ll]_-{p_2}\ar[rr]^-{p_1} & & B^p.
}
\]
Note that the total Chern class may be computed as (where $\pi_i:B^p \to B$ is the $i^{th}$ projection)
\begin{align*}
c(\xi^p)=& c(\pi_1^*(\xi)\oplus\cdots\oplus\pi^*_p(\xi)) \\
=& \prod_{i=1}^p c(\pi_i^*(\xi))\\
=& \prod_{i=1}^p (1+1\otimes\cdots\otimes c_1(\xi)\otimes \cdots \otimes 1)\\
=& 1+ c_1(\xi)\otimes \cdots \otimes c_1(\xi) + \sum_{\substack{\small q_i\in \{0,1\} \\(q_1,\cdots,q_p)\neq (0,\cdots,0) \\ (q_1,\cdots,q_p)\neq (1,\cdots,1)}}c_{q_1}(\xi)\otimes\cdots\otimes c_{q_p}(\xi).
\end{align*}
This gives \begin{align*}
p_2^*(c(\xi\wr C_p))=& p_1^*(c(\xi^p))\\
=& p_1^*\Big(1+  c_1(\xi)\otimes \cdots \otimes c_1(\xi) + \sum_{\substack{\small q_i\in \{0,1\} \\(q_1,\cdots,q_p)\neq (0,\cdots,0) \\ (q_1,\cdots,q_p)\neq (1,\cdots,1)}}c_{q_1}(\xi)\otimes\cdots\otimes c_{q_p}(\xi)\Big).\\
\end{align*}
The total Chern class of $\xi\wr C_p$ is concentrated in degrees $\leq 2p$. From Theorem \ref{cohwrpow}, we have that in these degrees, the cohomology of $[BU(1)]^p_{hC_p}$ is a sum of $P(c_1)$, monomials in $u$ and $v$, and elements in the image of $I$. The map $p_2$ is homotopic to the inclusion of the fibre in \eqref{F6}, therefore, $p_2^\ast c(\xi \wr C_p)$ computes the part of $c(\xi\wr C_p)$ that are not monomials in $u$ or $v$. It follows that 
\[
c(\xi\wr C_p)= 1+ P(c_1(\xi)) + z(c(\xi)) + \text{ sum of monomials in } u \text{ and } v.
\]
Let $x\in B$, be a point.  The inclusion  $i: x\xhookrightarrow{} B$ gives rise to a map between fibrations
\[
\xymatrix{
\xi_x\wr C_p \ar[rr]\ar[d] & & E^p_{hC_p} \ar[d]\\
BC_p \ar[rr]^{i\wr C_p} & & B^p_{hC_p}.\\
}
\]
Note that the pull-back $(i\wr C_p)^*(\xi\wr C_p)$ is the dimension $p$-bundle induced by the regular representation of $C_p$. In other words, let $\lambda$ denote one dimensional complex representation of $C_p$ where the chosen generator of $C_p$ acts by rotation of angle $\frac{2\pi}{p}$. The pull-back bundle has the following structure
\[\sum_{i=0}^{p-1}\xi_x\otimes \lambda^i.\] 
Thus the total Chern class of the pull-back bundle is
\begin{align*}
c(\xi_x\wr C_p)&=\prod_{i=0}^{p-1}c(\xi_x\otimes \lambda^i)
=\prod_{i=0}^{p-1} c(\lambda^i)
=\prod_{i=0}^{p-1} (1+\lambda v)
= 1+v^{p-1}.
\end{align*}
Therefore, 
\[(i\wr C_p)^*(c(\pi\wr C_p))=1+v^{p-1}\] 
which implies 
\[c(\xi\wr C_p)=(1+v^{p-1})+P(c_1(\xi))+z(c(\xi)), \]
proving the proposition.
\end{proof}

We now compute the Chern classes of $\xi\wr C_p$ for a general vector bundle $\xi$. 
\begin{theorem}\label{TCh}
Let $\xi$ be a $n$-dimensional complex vector bundle over a CW-complex $B$. Then
\begin{myeq} \label{chernwreath} 
c(\xi\wr C_p)=\sum_{0\leq r\leq n}P(c_r(\xi))(1+v^{p-1})^{n-r}+z(c(\xi)).
\end{myeq}  
\end{theorem}
\begin{proof}
We proceed by induction on the dimension of the bundle. Using the splitting principle, together with the relation of the wreath power with direct sums (Proposition \ref{P5}) and pull-backs (Proposition \ref{P6}), it suffices to assume the theorem for $\xi$ and prove it for $\xi\oplus L$ where $L$ is a line bundle. In this reduction, we are also using that an injective map on cohomology induces an injective map on wreath powers (Proposition \ref{P7}).  
Assume that $\dim(\xi)=n-1$, so that we have
\begin{align*}
 c((\xi\oplus L)\wr C_p)
&= c(\xi \wr C_p)c(L\wr C_p)\\
& = \Big[\sum_{0\leq r\leq n-1}P(c_r(\xi))(1+v^{p-1})^{n-1-r}+z(c(\xi))\Big] \Big((1+v^{p-1})+P(c_1(L))+z(c(L))\Big).
\end{align*}
The last formula comes from Proposition \ref{P8}. For the ease of doing the calculation, we denote  
\begin{myeq}\label{i}
\begin{aligned}
&\text{ I :$=$ } \sum_{0\leq r\leq n-1}P(c_r(\xi))(1+v^{p-1})^{n-1-r} \\
&\text{ II :$=$ } z(c(\xi))  \\
& \text{ III :$=$ } (1+v^{p-1})+P(c_1(L)) \\
& \text { IV :$=$ } z(c(L)),
\end{aligned}
\end{myeq}
so that in terms of \eqref{i},
\[  c((\xi\oplus L)\wr C_p) = \text{I} \cdot \text{III} + \text{I} \cdot \text{IV} + \text{II} \cdot \text{III} + \text{II} \cdot \text{IV}. \]
Since $P$ is multiplicative, I$\cdot$ III  gives 
\[\sum_{0\leq r \leq n-1}P(c_r(\xi))(1+v^{p-1})^{n-r}+\sum_{0\leq r \leq n-1}P(c_r(\xi)c_1(L))(1+v^{p-1})^{n-1-r}.\]
We write $Z'(x,y)=P(x+y)-P(x)-P(y)$ for two elements $x,y$ in the cohomology of $B$, and observe from 
\[c_r(\xi\oplus L)= c_r(\xi)+c_{r-1}(\xi)c_1(L)\]
that 
\begin{align*}
P(c_r(\xi\oplus L))&= \; \otimes_{p}(c_r(\xi)+c_{r-1}(\xi)c_1(L))\\
&=\; P(c_r(\xi))+P(c_{r-1}(\xi)c_1(L))+Z^{\prime}(c_r(\xi),c_{r-1}(\xi)c_1(L)).
\end{align*}
Note that $Z'(c_r(\xi),c_{r-1}(\xi)c_1(L))$ lies in the image of the operator $I$, so that it gives $0$ when multiplied by $v$. Thus,
\begin{align*}
\text{I$\cdot$ III} &=\sum_{0\leq r\leq n}\Big[P(c_r(\xi))+P(c_{r-1}(\xi)c_1(L))\Big](1+v^{p-1})^{n-r}\\
&= \sum_{0\leq r\leq n}P(c_r(\xi\oplus L))(1+v^{p-1})^{n-r}- \sum_{0\leq r\leq n}Z^{\prime}(c_r(\xi),c_{r-1}(\xi)c_1(L))\\
&= \sum_{0\leq r\leq n}P(c_r(\xi\oplus L))(1+v^{p-1})^{n-r}- \sum_{0\leq r\leq n}P(c_r(\xi\oplus L)) - P(c_r(\xi)) - P(c_{r-1}(\xi)c_1(L)) .
\end{align*}
For the rest of the products we use the fact that the image of the operator $z$ is contained in the image of $I$, so that $v$ multiplies to $0$. Therefore we have 
\begin{align*}
\text{I}\cdot \text{IV} & =  ~~  \sum_{0\leq r\leq n-1}P(c_r(\xi))z(1+c_1(L)). \\
\text{II}\cdot \text{III} & =  ~~z(c(\xi))+z(c(\xi))P(c_1(L)). \\
\text{II}\cdot \text{IV} & = ~~ z(c(\xi))z(1+c_1(L)).
\end{align*}
We simplify the first expression as 
\begin{align*} 
\text{I}\cdot \text{IV} = ~~ & \sum_{0\leq r\leq n-1}P(c_r(\xi))z(1+c_1(L)) \\
      & = \sum_{0\leq r\leq n-1}P(c_r(\xi))\Big( P(1+c_1(L)) - 1- P(c_1(L)) \Big) \\
          &= \sum_{0\leq r\leq n-1}P(c_r(\xi)(1+c_1(L)))-P(c_r(\xi)) - P(c_r(\xi)c_1(L)).
\end{align*}

Therefore, we get that 
\begin{myeq}\label{first2}
\text{I} \cdot \text{III} + \text{I} \cdot \text{IV} = \sum_{0\leq r\leq n}P(c_r(\xi\oplus L))(1+v^{p-1})^{n-r} +  \sum_{0\leq r\leq n-1}\Big(P(c_r(\xi)(1+c_1(L))) - P(c_{r+1}(\xi\oplus L)) \Big).
\end{myeq}
We next have 
\begin{align*}
\text{II}\cdot \text{III}  = ~~ & z(c(\xi))+z(c(\xi))P(c_1(L)) \\
& = \Big(P(c(\xi))- \sum_{0\leq r \leq n-1} P(c_r(\xi))\Big)(1+P(c_1(L)) \\ 
&= P(c(\xi)) + P(c(\xi)c_1(L)) -   \sum_{0\leq r \leq n-1} \Big(P(c_r(\xi)) + P(c_r(\xi)c_1(L))\Big). 
\end{align*}
The last expression gives 
\begin{align*} 
\text{II}\cdot \text{IV} & = ~~ z(c(\xi))z(1+c_1(L)) \\
& = \Big[ P(c(\xi)) - \sum_{0\leq r \leq n-1} P(c_r(\xi))\Big] (P(1+c_1(L))-1-P(c_1(L)))\\
&= P(c(\xi\oplus L))-P(c(\xi)) - P(c(\xi)c_1(L)) - \\
& \sum_{0\leq r \leq n-1}\Big[ P(c_r(\xi)(1+c_1(L)) - P(c_r(\xi)) - P(c_r(\xi)c_1(L))\Big].  
\end{align*}
Adding the last two expressions together we get 
\begin{myeq}\label{last2}
 \text{II} \cdot \text{III} + \text{II} \cdot \text{IV} = P(c(\xi\oplus L)) - \sum_{0\leq r \leq n-1} P(c_r(\xi)(1+c_1(L)).
\end{myeq}
Now summing all the terms together using \eqref{first2} and \eqref{last2}, we get  
\[
\begin{aligned}
&\text{I} \cdot \text{III} + \text{I} \cdot \text{IV} + \text{II} \cdot \text{III} + \text{II} \cdot \text{IV} \\
&=  P(c(\xi\oplus L)) + \sum_{0\leq r\leq n}P(c_r(\xi\oplus L))(1+v^{p-1})^{n-r} - \sum_{0\leq r\leq n-1} P(c_{r+1}(\xi\oplus L))  \\ 
 &=  \sum_{0\leq r\leq n}P(c_r(\xi\oplus L))(1+v^{p-1})^{n-r} + z(c(\xi\oplus L)). 
\end{aligned}
\]
%
This calculcation completes the proof of the theorem.
\end{proof}
\end{mysubsection}

\begin{mysubsection}{The Pontrjagin class of a wreath power}
The computations for the Pontrjagin classes of the $p$-th wreath power of a real  vector bundle $\xi : E\to B$ is now straightforward using Theorem \ref{TCh}.  Recall that $p_i(\xi) = (-1)^ic_{2i}(\xi\otimes \C)$. One readily observes 
\begin{myeq}\label{complexwreath}
(\xi\otimes \C)\wr C_p\cong (\xi\wr C_p)\otimes \C.
\end{myeq}
Now apply Theorem \ref{TCh} along with  \eqref{complexwreath} to deduce the following formula.
\begin{prop}\label{pontwreath}
The $\pmod p$ reduction of the Pontrjagin classes of a real $n$-dimensional vector bundle $\xi$ are given by,
\[
p_i(\xi\wr C_p)
=\Big[ \sum_{0\leq r\leq \floor{\frac{n}{2}}}(-1)^{r-i} P(p_r(\xi))(1+v^{p-1})^{n-2r}\Big]_{2i}+z_{4i}(p(\xi)).\]
\end{prop}

\begin{proof}
Theorem \ref{TCh} implies that 
\[c(\xi\wr C_p\otimes \C)
=\sum_{0\leq r\leq n}P(c_r(\xi\otimes \C))(1+v^{p-1})^{n-r}+z(c(\xi\otimes \C)).\]
We know that the odd Chern classes of complexification of real vector bundle are $2$-torsion \cite[Pg 174]{MS74}. Since we are working over $\Z/p$ for odd  $p$, the odd Chern classes $c_{2i+1}(\xi\otimes \C)=0$. This gives 
\begin{align*}
(-1)^ic_{2i}(\xi\wr C_p\otimes \C)
&=(-1)^i\big[\sum_{0\leq r\leq n}P(c_r(\xi\otimes \C))(1+v^{p-1})^{n-r}\big]_{2i}+(-1)^iz_{4i}(c(\xi\otimes \C))\\
\implies p_i(\xi\wr C_p)&=(-1)^i\big[\sum_{0\leq r\leq [\frac{n}{2}]}P(c_{2r}(\xi\otimes \C))(1+v^{p-1})^{n-2r}\big]_{2i}+(-1)^iz_{4i}(c(\xi\otimes \C))
\end{align*}
\begin{align*}
&=\big[\sum_{0\leq r\leq [\frac{n}{2}]}(-1)^{rp+i}P(p_{i}(\xi))(1+v^{p-1})^{n-2r}\big]_{2i}+ (-1)^i\sum_{\substack{\small q_i\in \{0,1,\cdots,[\frac{n}{2}]\} \\(q_1,\cdots,q_p)\neq (j,\cdots,j)\\ \sum q_j = i}} (-1)^{q_1}p_{q_1}\otimes\cdots\otimes (-1)^{q_p}p_{q_p}\\
&=\big[\sum_{0\leq r\leq [\frac{n}{2}]}(-1)^{r+i}P(p_{r}(\xi))(1+v^{p-1})^{n-2r}\big]_{2i}+z_{4i}(p(\xi)).
\end{align*}
The last equality comes from the fact that the signs on the second term cancel out and $(-1)^{rp}=(-1)^r$. 
\end{proof}

\end{mysubsection}

\begin{mysubsection}{The Euler class of a wreath power}
For computation of the$\pmod p$ reduction of the Euler class we use the generalized splitting principle for any $G$-vector bundle which states,  
\begin{theorem}\cite[Theorem 4]{May05}\label{split}
Let $T$ be a maximal torus of a compact connected Lie group $G$ of rank $n$. For every principal $G$-bundle $\xi: E \to X$, there exists $q: Y\to X$ with fibre $G/T$, and a reduction of the structure group of $q^*(\xi)$ to $T$ such that $H^*(Y)=H^*(X)\otimes H^*(G/T)$ and $q^*$ is the canonical inclusion.
\end{theorem}
Let $G$ be $SO(2n)$ and $T$ be $SO(2)^n$ embedded in $G$. Theorem \ref{split} implies that if $\xi$ is an oriented real $2n$-vector bundle then $q^*(\xi)$ splits as sum of $2$-plane bundles \cite[Example 10]{May05} with $q^*$ injective in cohomology. This enables us to compute the $\pmod p$ reduction of Euler class of $p$-th wreath power of a bundle by first computing the same for line bundles, and then  applying the generalized splitting principle.
\begin{prop} \label{eulerwreath}
The$\pmod p$ reduction of the Euler class of the wreath power of an even dimensional real  bundle $\xi$ is given by,
\[e(\xi\wr C_p)=P(e(\xi)).\]
\end{prop}
\begin{proof}
The proof starts at $2$-dimensional bundles $\xi$ which is analogous to  Proposition \ref{P8}. We observe that 
 \[e(\xi\wr C_p)=P(e(\xi))+ \text{ sum of monomials in } u \text{ and } v.
 \]
We form the following diagram for a point $x\in B$
\[
\xymatrix{
\xi_x\wr C_p \ar[rr]\ar[d] & & E^p_{hC_p} \ar[d]\\
BC_p \ar[rr]^{i\wr C_p} & & B^p_{hC_p},\\
}
\]
and notice that the pull-back $(i\wr C_p)^*(\xi\wr C_p)$ is the dimension $2p$-bundle induced by the regular representation of $C_p$. Since the regular representation has a trivial summand, the Euler class is zero, and hence, the monomial part in $u$ and $v$ does not contribute. This proves the Proposition for $2$-dimensional bundles.

We proceed by induction on $n$ where the dimension of the bundle is $2n$. Using the generalised splitting principle described above it suffices to assume the theorem for $\xi$ and prove it for $\xi\oplus L$ where $L$ is a $2$-plane bundle. In this case,
\begin{align*}
e(\xi\oplus L\wr C_p)&=e(\xi\wr C_p)e(L\wr C_p)\\
&=P(e(\xi))P(e(L))\\
&=P(e(\xi)(e(L))\\
&=P(e(\xi\oplus L)),
\end{align*}
which completes the proof.
\end{proof}
\end{mysubsection}

\section{Index computations}\label{indcomp}

In \S \ref{wrpowsp}, the index computation for $F_n^G$ was reduced to computing the pull-back of appropriate characteristic classes (Proposition \ref{indred}).   These pull-backs are the universal characteristic classes for the $p$-fold wreath powers in the notation of \S \ref{wreathcharclass}.  
\begin{mysubsection}{The unitary case}
We denote the $n$-dimensional universal bundle by $\gamma^n_{G}$ over $G(n)$ where $G$ could be $U$ or $SO$. Observe that in the notation of the diagram \eqref{CD1},
  \[BW_n^G \simeq BG(n)^p_{hC_p},~~~i_B^*(\gamma^n_G) = \gamma^n_G \wr C_p. \]
In the spectral sequence for the fibration $G(n)\to EG(n) \to BG(n)$, the cohomology of $G(n)$ is an exterior algebra on transgressive elements, and the image of the transgression are the universal characteristic classes. Using this we can figure out the kernel of $p_1^\ast$ in \eqref{CD1}, which for $G=U$ yields
%
\begin{myeq}\label{Chin}
\Ker(p_1^*) = \Big( c_1(\gamma^n_{U}\wr C_p),\cdots,c_{pn}(\gamma^n_{U}\wr C_p)\Big). 
\end{myeq}
We first simplify the notation to be used in the computations below. 
\begin{notation}\label{zk}
We denote $c_k(\gamma_{U}^n)$ by $c_k$, $p_k(\gamma_{SO}^n)$ by $p_k$, and $e_k(\gamma_{SO}^n)$ by $e_k$. Recall the notation $z(\phi)$ for $\phi\in H^\ast(X)$ from Notation \ref{z}. We denote the degree $k$ part of $z(\phi)$ as $z_k(\phi)$. For a formal sum of cohomology classes $\phi$ we denote the degree $2k$ part of $\phi$ by $[\phi]_k$.
\end{notation}
Our target is to find out the value of $l$ as defined in Proposition \ref{indred}. We start with the easy case where $p\nmid n$ which is done by a direct computation. 
\begin{theorem} \label{relprime}
Suppose that $p\nmid n$. Then the index of $F_n^U$ is given by the formula 
\[\Index_{C_p}(F_n^U)=
(uv^{p-1},v^{p}). \]
\end{theorem}
\begin{proof}
First we consider the case for $n=1$. Using the equation \eqref{chernwreath}, we compute the Chern classes of $p$-th wreath power of $\gamma_U^n$ as
 \[
c(\gamma_{U}^n\wr C_p)=\sum_{0\leq r\leq n}P(c_r)(1+v^{p-1})^{n-r}+z(c),
\]
where $c=1+c_1$. Therefore, the $i$-th Chern classes are 
\begin{align*}
&c_1(\gamma_{U}^1\wr C_p) =I(c_1\otimes 1 \otimes \cdots \otimes 1),\\
&c_2(\gamma_{U}^1\wr C_p) =\sum I(c_1\otimes 1 \otimes \cdots \otimes c_1 \otimes 1 \otimes \cdots \otimes 1) ,\\
&\cdots,\\
&c_{p-1}(\gamma_{U}^1\wr C_p) = v^{p-1}+z_{2(p-1)}(c),\\
&\cdots.
\end{align*}
Multiplying $c_{p-1}$ respectively by $u$ and $v$ we get $uc_{p-1}=uv^{p-1}$ and $vc_{p-1}=v^p$. Here we are using the fact that multiplying $z_k(c)$ with $u$ and $v$ yields zero. Note that $z_{2(p-1)}(c)$ is not zero modulo $z_{2i}(c)$ for $i\leq p-2$ by using the fact that the elementary symmetric polynomials are algebraically independent and the expressions of Notation \ref{z}. 
Therefore, $p-1$ is the lowest among $l$ such that $\pi^*(uv^{l-1})$ and $\pi^*(v^{l})$ are in $\Ker(p_1^*)\cap \Im(\pi^*)$. So by Proposition \ref{indred} we get $\Index_{C_p}(F_1^U)$ is  $(uv^{p-1},v^p)$. 

Now we turn to the case when $n\neq 1$. We again compute the Chern classes,
\begin{align*}
&c_1(\gamma_{U}^n\wr C_p) =z_2(c),\\
&c_2(\gamma_{U}^n\wr C_p) =z_4(c),\\
&\cdots,\\
&c_{p-1}(\gamma_{U}^n\wr C_p) = v^{p-1}+z_{2(p-1)}(c),\\
&\cdots.
\end{align*}
Proceeding as for the case $n=1$, we see that $\pi^*(uv^{p-1})$ and $\pi^*(v^p)$ are in $\Ker(p_1^*)\cap \Im(\pi^*)$. We will show this is the smallest degree term in $\Ker(p_1^*)\cap \Im(\pi^*)$. Observe that there is a $C_p$-equivariant map \[i: F_1^U\to F_n^U\]\[(V_1,V_2, \cdots, V_p)\mapsto (\oplus_n V_1, \cdots,  \oplus_n V_p). \] This gives that $\Index_{C_p}(F_n^U)\subset \Index_{C_p}(F_1^U)=(uv^{p-1},v^p)$ implying the theorem.
\end{proof}

The following proposition serves as a key step in determining $l$ such that $uv^{l-1}$ or $v^l$ are $0$ in the quotient algebra $H^\ast(BW_n^U)/\Ker(p_1^\ast)$ in the case $p\mid n$. 

\begin{prop}\label{5p}
Suppose that $p\mid n$ and write $n=p^aq$ for $a\geq 1$ and $p\nmid q$. Then for every $1\leq k\leq p^{a+1}-1$, the relation $c_k(\gamma^n_{U}\wr C_p)=0$ in the quotient algebra 
\[H^*(BW_n^U)/\Big( c_1(\gamma^n_{U}\wr C_p),\cdots,c_{k-1}(\gamma^n_{U}\wr C_p) \Big)\] 
is equivalent to the following relations:
\vspace{.5cm}
\begin{compactenum}[\rm (i)]
\item\label{5p1} If $p\nmid k$ and $1\leq k\leq p^{a+1}-2$, then 
\begin{myeq}\label{ind1}
z_{2k}(c)=0.
\end{myeq}
\item\label{5p2} If $p\mid k$ and $k\not\in \{p^{a+1}-p^a,\cdots, p^{a+1}-p\}$, then 
\begin{myeq}\label{ind2}
P(c_{\frac{k}{p}})+z_{2k}(c)=0.
\end{myeq}
\item\label{5p3} If $k=p^{a+1}-p^{m+1}$, where $m\in \{0,1\cdots,a-1\}$, then
\begin{myeq}\label{ind3}
P(c_{p^a-p^m})+\alpha_m v^{p^{a+1}-p^{m+1}}+z_{2k}(c)=0, \quad
\text{where $\alpha_m\in \Z/p^{\times}$.}
\end{myeq}
\item\label{5p4} If $k=p^{a+1}-1$, then
\begin{myeq}\label{ind4}
v^{p-1}P(c_{p^a-1})+z_{2(p^{a+1}-1)}(c)=0. 
\end{myeq}
\end{compactenum}
\end{prop}
\begin{proof}
The proof is analogous to \cite[Proposition 6.1]{BBKV18}. We proceed by induction on $k$ using the equation \eqref{chernwreath}
\begin{myeq}\label{chwr}
c(\gamma_{U}^n\wr C_p)=\sum_{0\leq r\leq n}P(c_r)(1+v^{p-1})^{n-r}+z(c).
\end{myeq} 
This implies 
\[
\begin{aligned}
c_k(\gamma_{U}^n\wr C_p)&=\Big[\sum_{0\leq r\leq n}P(c_r)(1+v^{p-1})^{n-r}+z(c)\Big]_k\\
                                         &= \sum_{0\leq r\leq n}\Big[ P(c_r)(1+v^{p-1})^{n-r}\Big]_k+z_{2k}(c).
\end{aligned}
\]
This already explains the terms $z_{2k}(c)$ occurring in the proposition. We show by induction that for $k\leq p^{a+1}-1$, $\sum_{0\leq r\leq n}\Big[ P(c_r)(1+v^{p-1})^{n-r}\Big]_k$ equals the other terms claimed in the four parts of the proposition modulo $c_i(\gamma_U^n \wr C_p)$ for $i\leq k-1$.  

Start with the case $p\nmid k$ and $k\leq (p-1)p^a-1$. As $p$ is an odd prime, $v^{p-1}$ is in degree $2(p-1)\geq 4$, so that the proposition is clear when $k=1$.  Now we compute the degree $2k$ part of \eqref{chwr}. We note that   
\begin{myeq} \label{1term}
\begin{aligned}
\Big[ (1+v^{p-1})^n \Big]_{k}&=\Big[ (1+v^{p-1})^{p^aq}\Big]_k\\
 &=\Big[(1+v^{(p-1)p^{a}})^q \Big]_k\\
 &= \sum_{j=0}^{q} \Big[\binom{q}{j}(v^{(p-1)p^a})^j \Big]_k
\end{aligned}
\end{myeq}
which does not have any degree $2k$ part as $(p-1)p^a >k$. Now examine $(1+v^{p-1})^{n-r}P(c_r)$ for each $r\geq 1$. In order to contribute to the degree $2k$ part of \eqref{chwr}, we must have $rp<k$, so that we may apply \eqref{ind2} and replace $P(c_r)$ with $-z_{2rp}(c)$. Consequently,
\[
(1+v^{p-1})^{n-r}P(c_r)= -(1+v^{p-1})^{n-r}z_{2rp}(c)
= -z_{2rp}(c) \quad \text{by \eqref{Einfty},}
\]
which does not contribute in degree $2k$, proving \eqref{5p1} in this case.

Now suppose $p\mid k$ and $k\leq (p-1)p^a -1$. An entirely analogous argument as above shows that the 
degree $2k$ part of $(1+v^{p-1})^{n-r}P(c_r)$ may be non-zero only when $r=k/p$. In this case 
\[
\Big[(1+v^{p-1})^{n-\frac{k}{p}}P(c_{\frac{k}{p}}) \Big]_{k} = P(c_{\frac{k}{p}}) \]
which proves \eqref{5p2} in this case. 
  
Now consider $k=(p-1)p^a = p^{a+1} - p^a$. Now \eqref{1term} changes to 
\[ 
\begin{aligned}
\Big[ (1+v^{p-1})^n \Big]_{k}
 &= \sum_{j=0}^{q} \Big[\binom{q}{j}(v^{(p-1)p^a})^j \Big]_k \\
&= qv^k = \alpha_{a-1} v^k,
\end{aligned}
\]
where $\alpha_{a-1} \equiv q \pmod{p}$. The rest of the argument proceeds as above for the case \eqref{5p2} to imply \eqref{5p3} for $m=a-1$. 

Next we prove \eqref{5p1} for $(p-1)p^a < k < p^{a+1}-1$. The equation \eqref{1term} still does not yield any term in degree $2k$, as $p\nmid k$. Also unless $rp$ is one of $p^{a+1}-p^{m+1}$ for $0\leq m \leq a-1$, the term $P(c_r)(1+v^{p-1})^{n-r}$ does not contribute in degree $k$ as in the cases above. Finally if  $rp=p^{a+1}-p^{m+1}$ where $0\leq m\leq a-1$, we may apply \eqref{ind3}
to obtain 
\begin{myeq}\label{degklarge}
\begin{aligned}
 \Big[ P(c_r)(1+v^{p-1})^{n-r} \Big]_k &= \Big[(-z_{rp}(c)-\alpha_m v^{rp})(1+v^{p-1})^{n-r}\Big]_k\\
 &= \Big[-z_{rp}(c)-\alpha_m v^{rp}(1+v^{p-1})^{n-r}\Big]_k\\
&= \begin{cases} -\alpha_m \binom{n-r}{\frac{k-rp}{p-1}} v^k & \mbox{ if } p-1\mid k-rp\\
 0 & \mbox{ if } p-1\nmid k-rp. \end{cases}
\end{aligned}
\end{myeq}
Hence, we are done unless $(p-1)\mid k-rp$. Now suppose that $k=rp+\lambda (p-1)$. Observe that $n-r=p^aq-p^a+p^m$. This implies $p^m\mid n-r$. On the other hand, $p\nmid k$ implies that $p\nmid \lambda$. Therefore, $\binom{n-r}{\lambda}=0$ unless $m=0$. But this means that 
\[
k-rp < p^{a+1}-1 - (p^{a+1}-p ) < p-1.\]
Thus, we end up with \eqref{ind1}. 

If $p\mid k$ for $(p-1)p^a < k < p^{a+1}-1$, and $k\neq p^{a+1}-p^{m+1}$ for $0\leq m \leq a-1$, we work out that \eqref{1term} does not have any elements in degree $2k$ as there are no divisors of $(p-1)p^a$ in the given range of values of $k$. Proceeding as in the previous case, we only need to figure out the degree $2k$ part of $P(c_r)(1+v^{p-1})^{n-r}$ where $rp<k$ is of the form $p^{a+1}-p^{m+1}$ for $0\leq m \leq a-1$. The formula \eqref{degklarge} still holds, and here $\lambda=\frac{k-rp}{p-1}$ is divisible by $p$. However $p^m \mid n-r$ implies $\lambda = j\cdot p^m$. If $j=1$, 
\[
k=rp+p^m(p-1)= p^{a+1}-p^{m+1} + p^{m+1}-p^m= p^{a+1}-p^m,\]
which contradicts the assumption on $k$, and if $j\geq 2$, 
\[
k=p^{a+1}-p^m+p^m(p-1)(j-1)>p^{a+1}-2. \]
This completes the proof of \eqref{5p2}. 

Now we consider $k=p^{a+1}-p^{m+1}$, where $0\leq m \leq a-1$, and note that $\frac{k}{p}=p^a-p^m$. We have already completed the calculation for $m=a-1$. For the remaining, we check that \eqref{1term} does not contribute in degree $2k$ as $(p-1)p^a \nmid k$. As in the above cases, the contributions in degree $2k$ apart from $z_{2k}(c)$ and $P(c_{p^a-p^m})$ may only come from terms $P(c_r)(1+v^{p-1})^{n-r}$ for $r=p^a-p^s$ where $a-1\geq s>m$. In this case, 
\begin{align*}
 \Big[P(c_r)(1+v^{p-1})^{n-r}\Big]_k &= \Big[(-z_{2rp}(c)-\alpha_s v^{rp})(1+v^{p-1})^{n-r}\Big]_k\\
 &= \Big[-z_{2(p^{a+1}-p^{s+1})}(c)-\alpha_s v^{p^{a+1}-p^{s+1}}(1+v^{p-1})^{p^aq-p^a + p^s}\Big]_k \\
&= -\alpha_s \binom{p^a q -p^a+p^s}{\frac{p^{s+1}-p^{m+1}}{p-1}}v^k.
\end{align*}
Note that $p^s$ divides $p^a q -p^a+p^s$, while the highest power of $p$ dividing $\frac{p^{s+1}-p^{m+1}}{p-1}$ is $p^{m+1}$. Therefore, this is $0$ unless $s=m+1$, in which case the binomial coefficient is $\binom{p^a q -p^a+p^{m+1}}{p^{m+1}}$ which is $\not\equiv 0 \pmod{p}$ by Lucas'
theorem \cite{Lu78}.

Finally, suppose $k=p^{a+1}-1$. As in the above cases other than $z_{2k}(c)$, the only contributing terms in degree $2k$ in \eqref{chwr} may arise in $P(c_r)(1+v^{p-1})^{n-r}$ for $rp=p^{a+1}-p^{m+1}$ ($0\leq m\leq a-1$). As in the cases above, we have 
\begin{align*}
 \Big[P(c_r)(1+v^{p-1})^{n-r}\Big]_{p^{a+1}-1} &= \Big[(-z_{2rp}(c)-\alpha_m v^{rp})(1+v^{p-1})^{n-r}\Big]_{p^{a+1}-1}\\
 &= \Big[-z_{2(p^{a+1}-p^{m+1})}(c)-\alpha_m v^{p^{a+1}-p^{m+1}}(1+v^{p-1})^{p^aq-p^a + p^m}\Big]_{p^{a+1}-1} \\
&= -\alpha_m \binom{p^a q -p^a+p^m}{\frac{p^{m+1}-1}{p-1}}v^{p^{a+1}-1},
\end{align*}
which is $0$ unless $m=0$, as $p\mid p^a q -p^a+p^m$ but $p\nmid \frac{p^{m+1}-1}{p-1}$. At $m=0$, we have
\begin{align*}
 \Big[P(c_{p^a -1})(1+v^{p-1})^{n-p^a+1}\Big]_{p^{a+1}-1} 
&= (p^a(q-1)+1)P(c_{p^a-1})v^{p-1}\\
 &= P(c_{p^a-1})v^{p-1}.
\end{align*}
This implies \eqref{ind4} and completes the proof. 
\end{proof}

The computations of Proposition \ref{5p} allow us to deduce 
\begin{cor}\label{5c1}
Let $n=p^aq$, where $p \nmid q$. Then in the quotient algebra $H^*(BW_n^U)/\Big( c_1(\gamma^n_{U}\wr C_p),\cdots,c_{p^{a+1}-1}(\gamma^n_{U}\wr C_p) \Big)$ we have the following two relations. \[uv^{p^{a+1}-1}=0,\]\[v^{p^{a+1}}=0.\]
\end{cor}
\begin{proof}
From the Proposition  \eqref{5p}\eqref{5p3} and \eqref{5p}\eqref{5p4} we have following two equations.
\begin{myeq}\label{cor1}
 v^{p-1}P(c_{p^a-1})+z_{2(p^{a+1}-1)}(c)=0
\end{myeq}
\begin{myeq}\label{cor2}
P(c_{p^a-1})+\alpha_0 v^{p^{a+1}-p}+z_{2(p^{a+1}-p)}(c)=0
\end{myeq}
where $\alpha_0\in \Z/p^\times$. Multiplying \eqref{cor2} with $v^{p-1}$ and then subtracting from  \eqref{cor1} we get 
\begin{myeq}\label{cor3}
\alpha_0 v^{p^{a+1}-1}-z_{2(p^{a+1}-1)}(c)=0.
\end{myeq}
Multiplying the above equation further with $u$ and $v$ gives our desired result. Here we are using the fact $u\cdot z_k(c)=0$ and $v\cdot z_k(c)=0$.
\end{proof}
We are now in a position to compute the index in the case $G=U$. 
\begin{thm} \label{Findex}
Suppose $n=p^a q$ for $a\geq 1$, and $p\nmid q$. Then,
\[\Index_{C_p}(F^U_n)=( uv^{(p^{a+1}-1)}, v^{p^{a+1}}).\]
\end{thm}
\begin{proof}
Let us first prove the theorem for the case $p^a$. From Proposition \ref{indred}, we know that the index is $( uv^{k-1}, v^k)$ where $\pi^*(uv^{k-1})$ or $\pi^*(v^{k})$ is the lowest degree non-zero element in 
\[
H^*(BW_n^U)/\Big( c_1(\gamma^n_{U}\wr C_p),\cdots,c_{pn}(\gamma^n_{U}\wr C_p)\Big) \cap \Im(\pi^*).
\]
By corollary \eqref{5c1} we already have $uv^{p^{a+1}-1}=0$ and $v^{p^{a+1}}=0$. If we can show 
\[v^{p^{a+1}-1}\not\in \Big( c_1(\gamma^n_{U}\wr C_p),\cdots,c_{pn}(\gamma^n_{U}\wr C_p)\Big)\] 
and
\[uv^{p^{a+1}-2}\not\in \Big( c_1(\gamma^n_{U}\wr C_p),\cdots,c_{pn}(\gamma^n_{U}\wr C_p)\Big)\] 
we are done. 
 Let us take a bigger ideal 
\[J=\Big( \{c_1(\gamma^n_{U}\wr C_p),\cdots,c_{pn}(\gamma^n_{U}\wr C_p)\}\cup \{I(c_{q_1}\otimes c_{q_2}\cdots\otimes c_{q_{p}})\mid 1 \leq \sum_{1}^{p} q_i \leq p^{a+1}-p\} \Big)\] where $I$ is as described in \eqref{E2}
and show $v^{p^{a+1}-1}$ is not in $J$. 
Now \eqref{cor3} gives
\[v^{p^{a+1}-1}=\alpha^{\prime}z_{2(p^{a+1}-1)}(c),\]
where $\alpha^\prime = \alpha_0^{-1}$. Expressing $z_{2(p^{a+1}-1)}$ in terms of $I$,
\begin{align*}z_{2(p^{a+1}-1)}&=\sum_{\sum q_i=p^{a+1}-1} c_{q_1}\otimes \cdots \otimes c_{q_p}\\
&=I(c_{p^a-1}\otimes c_{p^a}\cdots\otimes c_{p^a}).
\end{align*} which is non-zero in $H^*(BW_n^U)/J$. Thus we have proved $v^{p^{a+1}-1}\neq 0$ in $H^*(BW_n^U)/\Big( c_1(\gamma^n_{U}\wr C_p),\cdots,c_{pn}(\gamma^n_{U}\wr C_p)\Big)$ proving the theorem for $n=p^a$.\\
For the general case $n=p^aq$ where $p\nmid q$, consider the $C_p$-equivariant map \[F_{p^a}^U\to F_{p^aq}^U\]
\[(V_1,V_2,\cdots,V_p)\mapsto (\oplus_{q}V_1,\cdots,\oplus_{q}V_p).\] The existence of such $C_p$-equivariant map will guarantee $\Index_{C_p}F_{p^aq}^U\subset \Index_{C_p}F_{p^a}^U$. Thus by Corrollary \eqref{5c1} we conclude the theorem.
\end{proof}
\end{mysubsection}

\begin{mysubsection}{The real case}
The computation of $\Index_{C_p}F_n^{SO}$ is analogous to its complex counterpart $\Index_{C_p}F_n^{U}$. Observe that for $G=SO$ the kernel of $p_1^\ast$ in \eqref{CD1} is
\begin{myeq}\label{ptin}
 \Ker(p_1^*)\cong 
 \begin{cases} \Big(p_1(\gamma^n_{SO}\wr C_p),\cdots,p_{\frac{pn}{2}}(\gamma^n_{SO}\wr C_p),e_{pn}(\gamma^n_{SO}\wr C_p)\Big) & \text{if $n$ is even},\\
 \Big( p_1(\gamma^n_{SO}\wr C_p),\cdots,p_{\frac{pn}{2}-1}(\gamma^n_{SO}\wr C_p)\Big) & \text{if $n$ is odd.}
\end{cases}
\end{myeq}
For a real even dimensional bundle the Euler class appears in the index of its $p$-th wreath power. This is where the computation could have been different from the complex case, however, a simple comparison allows us to prove this case.
\begin{theorem}\label{indeven}
Let $n$ be even and $n=p^a q$ with $a\geq 0$, $p\nmid q$. Then,
\[\Index_{C_p}(F^{SO}_n)=( uv^{(p^{a+1}-1)}, v^{p^{a+1}}).\]
\end{theorem}
\begin{proof}
Consider the maps between flag manifold 
\[F_{\frac{n}{2}}^{U}\to F_{n}^{SO}\to F_{n}^{U}\] 
where the first map is the underlying real subspace of complex subspaces and the second map is the complexification. This will imply \[\Index_{C_p}F_{n}^U\subset \Index_{C_p}F_{n}^{SO}\subset \Index_{C_p}F_{\frac{n}{2}}^U.\]But we already have $\Index_{C_p}F_{n}^U=\Index_{C_p}F_{\frac{n}{2}}^U=(uv^{(p^{a+1}-1)}, v^{p^{a+1}})$ which forces $\Index_{C_p}F_{n}^{SO}$ to be $(uv^{(p^{a+1}-1)}, v^{p^{a+1}})$.
\end{proof}
Now we will turn to the case when $n$ is odd. The proof of Theorem \ref{indeven} already shows that if $n=p^a q$ with $p\nmid q$, 
\[\Index_{C_p}(F^{SO}_n)\subset ( uv^{(p^{a+1}-1)}, v^{p^{a+1}}),\]
using the map $F_n^{SO} \to F_n^U$. First we investigate what happens when $n$ does not involve any non-trivial power of $p$.
\begin{theorem} \label{relprimeso}
Suppose that $p\nmid n$. The index of $F_n^{SO}$ is given by the formula 
\[\Index_{C_p}(F_n^{SO})=\begin{cases} 
(v^{p-1})  & \text{if $n=1$},\\
(uv^{p-1},v^{p}) & \text{if $n> 1$}.\end{cases}\]
\end{theorem}
\begin{proof}
In the case for $n=1$, $F_n^{SO} \simeq SO(p)$, and $W_1^{SO}\simeq C_p$. The total Pontrjagin class of the wreath power equals the total Pontrjagin class of the regular representation from which we obtain 
%
$\Index_{C_p}(F_1^{SO})$ is  $(v^{p-1})$. 

Now we turn to the case when $n> 1$. We compute according to \eqref{ptin} the Pontrjagin classes using Proposition \ref{pontwreath},
\begin{align*}
&p_1(\gamma_{SO}^n\wr C_p) =z_4(p(\gamma_{SO}^n)),\\
&p_2(\gamma_{SO}^n\wr C_p) =z_8(p(\gamma_{SO}^n)),\\
&\cdots,\\
&p_{\frac{p-1}{2}}(\gamma_{SO}^n\wr C_p) = v^{p-1}+z_{2(p-1)}(p(\gamma_{SO}^n)),\\
&\cdots.
\end{align*}
Let $p= 1+p_1+\cdots + p_{\frac{n-1}{2}}$. We observe now that if $n>1$, $z_{2(p-1)}(p)$ does not belong to the ideal generated by $z_4(p), \cdots , z_{2(p-1)-4}(p)$. Suppose on the contrary that there is a relation 
\[ 
z_{2(p-1)}(p) = \sum_{i=1}^{\frac{p-3}{2}} \lambda_i z_{4i}(p), 
\]
We may put $p_i=0$ for $i\geq 2$ in the expression above and obtain a relation among the homogeneous terms in $z_{4i}(1+p_1)$. These are elementary symmetric polynomials as observed in Notation \ref{z}, and so we obtain a contradiction. 
%
%
\end{proof}
To compute the index for general $n$ odd we need results analogous to Proposition \ref{5p} for the Pontrjagin classes. The proofs are similar so we skip them and state the result below. 
\begin{prop}\label{5pp}
Suppose that $p\mid n$, and write $n=p^aq$ for $a\geq 1$ and $p \nmid q$. For every $1\leq 2k\leq p^{a+1}-1$, the relation $p_k(\gamma^n_{SO}\wr C_p)=0$ in the quotient algebra 
\[H^*(BW_n^{SO})/\Big( p_1(\gamma^n_{SO}\wr C_p),\cdots,p_{k-1}(\gamma^n_{SO}\wr C_p) \Big)\] 
is equivalent to the following relations:
\vspace{.5cm}
\begin{compactenum}[\rm (i)]
\item\label{5pp1} If $p\nmid k$ and $1\leq 2k\leq p^{a+1}-2$, then 
\begin{myeq}\label{pind1}
z_{4k}(p)=0.
\end{myeq}
\item\label{5pp2} If $p\mid k$ and $2k\not\in \{p^{a+1}-p^a,\cdots, p^{a+1}-p\}$, then 
\begin{myeq}\label{pind2}
P(p_{\frac{k}{p}})+z_{4k}(p)=0.
\end{myeq}
\item\label{5pp3} If $2k=p^{a+1}-p^{m+1}$, where $m\in \{0,1\cdots,a-1\}$, then
\begin{myeq}\label{pind3}
P(p_{\frac{p^a-p^m}{2}})+\alpha_m v^{(p^{a+1}-p^{m+1})}+z_{4k}(p)=0, \quad
\text{where $\alpha_m\in \Z/p^{\times}$.}
\end{myeq}
\item\label{5pp4} If $2k=p^{a+1}-1$, then
\begin{myeq}\label{pind4}
v^{(p-1)}P(p_{\frac{p^a-1}{2}})+z_{2(p^{a+1}-1)}(p)=0. 
\end{myeq}
\end{compactenum}
\end{prop}
Using Proposition \ref{5pp} we complete the computation $\Index_{C_p}(F_n^{SO})$ for general odd $n$.
\begin{theorem}\label{indexfso}
Let $n$ be odd and $n=p^a q$ with $a\geq 1$, $p\nmid q$. Then,
\[\Index_{C_p}(F^{SO}_n)=\begin{cases} (v^{p^{a+1}-1}) & \text{if $q=1$},\\
                                                (uv^{p^{a+1}-1},v^{p^{a+1}}) & \text{if $q>1$}.\end{cases}\]
\end{theorem}
\begin{proof}
Let us first prove the theorem for $n=p^a$. This is analogous to Theorem \ref{Findex}. Taking $m=0$ in \eqref{pind3} and multiplying it with $v^{p-1}$ and subtracting it from \eqref{pind4} we get 
\[v^{p^{a+1}-1}=\alpha_0^\prime z_{2(p^{a+1}-1)}(p)\]
 where $\alpha^\prime_0=\alpha_0^{-1}$. However, observe 
\[
z_{2(p^{a+1}-1)}(p)=\sum_{\sum q_i=\frac{p^{a+1}-1}{2}} p_{q_1}\otimes \cdots \otimes p_{q_p},\]
is $0$ for degree reasons, as the maximum value each $q_i$ can attain is  $\frac{p^a-1}{2}$. Thus, $v^{p^{a+1}-1}=0$ in $H^*(BW_n^{SO})/\Big( p_1(\gamma^n_{SO}\wr C_p),\cdots,p_{\frac{pn-1}{2}}(\gamma^n_{SO}\wr C_p)\Big)$. Now if we can show  $v^{p^{a+1}-2}\not\in \Big( p_1(\gamma^n_{SO}\wr C_p),\cdots,p_{\frac{pn-1}{2}}(\gamma^n_{SO}\wr C_p)\Big)$ we are done for the case $p^a$. From  \eqref{pind3} we have 
\[v^{p^{a+1}-2}=\alpha_0^\prime v^{p-2}P(p_{\frac{p^a-1}{2}})\]  
This is clearly non-zero, as the only relation on $P(p_{\frac{p^a-1}{2}})$ in degrees $\leq 2(p^{a+1}-2)$ is given by \eqref{pind3} for $m=0$. 

For the general case, where $q>1$ we already have  
\[v^{p^{a+1}}=0,~~~uv^{p^{a+1}-1}=0,\] 
from the proof of Theorem \ref{indeven}. The proof is complete once we observe that 
\[v^{p^{a+1}-1}=\alpha_0^\prime z_{2(p^{a+1}-1)}(p)\]
and $z_{2(p^{a+1}-1)}(p)$ is non-zero modulo  the ideal $\Big( p_1(\gamma^n_{SO}\wr C_p),\cdots,p_{\frac{pn-1}{2}}(\gamma^n_{SO}\wr C_p)\Big)$.
We have
\[ z_{2(p^{a+1}-1)}(p)=\sum_{\sum q_i=\frac{p^{a+1}-1}{2}} p_{q_1}\otimes \cdots \otimes p_{q_p},\]
contains a term $I(p_{\frac{p^a+1}{2}} \otimes \cdots p_{\frac{p^a+1}{2}} \otimes p_{\frac{p^a-p}{2}})$ which may be used to show that the above expression is non-zero by arguments similar to the ones in the proof of Theorem \ref{relprimeso}. Hence the theorem follows for $q>1$.
\end{proof}
\end{mysubsection}


\begin{thebibliography}{10}

\bibitem{AM04}
{\sc A.~Adem and R.~J. Milgram}, {\em Cohomology of finite groups}, vol.~309 of
  Grundlehren der mathematischen Wissenschaften [Fundamental Principles of
  Mathematical Sciences], Springer-Verlag, Berlin, second~ed., 2004.

\bibitem{BBKV18}
{\sc D.~Barali\'{c}, P.~V.~M. Blagojevi\'{c}, R.~Karasev, and
  A.~Vu\v{c}i\'{c}}, {\em Index of {G}rassmann manifolds and orthogonal
  shadows}, Forum Math., 30 (2018), pp.~1539--1572.

\bibitem{BG21}
{\sc S.~Basu and S.~Ghosh}, {\em Bredon cohomology of finite dimensional
  {$C_p$}-spaces}, Homology Homotopy Appl., 23 (2021), pp.~33--57.

\bibitem{BaKu21}
{\sc S.~Basu and B.~Kundu}, {\em The index of certain stiefel manifolds}, 2021.

\bibitem{Bor53}
{\sc A.~Borel}, {\em Sur la cohomologie des espaces fibr\'{e}s principaux et
  des espaces homog\`enes de groupes de {L}ie compacts}, Ann. of Math. (2), 57
  (1953), pp.~115--207.

\bibitem{Fahu88}
{\sc E.~Fadell and S.~Husseini}, {\em An ideal-valued cohomological index
  theory with applications to {B}orsuk-{U}lam and {B}ourgin-{Y}ang theorems},
  Ergodic Theory Dynam. Systems, 8$^*$ (1988), pp.~73--85.

\bibitem{Le97}
{\sc I.~J. Leary}, {\em On the integral cohomology of wreath products}, J.
  Algebra, 198 (1997), pp.~184--239.

\bibitem{Lu78}
{\sc E.~Lucas}, {\em Theorie des {F}onctions {N}umeriques {S}implement
  {P}eriodiques}, Amer. J. Math., 1 (1878), pp.~289--321.

\bibitem{Mat03}
{\sc J.~Matou\v{s}ek}, {\em Using the {B}orsuk-{U}lam theorem}, Universitext,
  Springer-Verlag, Berlin, 2003.
\newblock Lectures on topological methods in combinatorics and geometry,
  Written in cooperation with Anders Bj\"{o}rner and G\"{u}nter M. Ziegler.

\bibitem{May05}
{\sc J.~P. May}, {\em A note on the splitting principle}, Topology Appl., 153
  (2005), pp.~605--609.

\bibitem{MS74}
{\sc J.~W. Milnor and J.~D. Stasheff}, {\em Characteristic classes}, Annals of
  Mathematics Studies, No. 76, Princeton University Press, Princeton, N. J.;
  University of Tokyo Press, Tokyo, 1974.

\bibitem{Oz87}
{\sc M.~Ozaydin}, {\em Equivariant maps for the symmetric group}, available at
  https://minds.wisconsin.edu/bitstream/handle/1793/63829/Ozaydin.pdf,  (1987).

\bibitem{Vol00}
{\sc A.~Y. Volovikov}, {\em On the index of {$G$}-spaces}, Mat. Sb., 191
  (2000), pp.~3--22.

\end{thebibliography}
\end{document}